\documentclass[12pt]{amsart}

\usepackage{amsxtra,amssymb,amsmath,amscd,url,listings}

\usepackage[alphabetic]{amsrefs}
\usepackage[utf8]{inputenc}
\usepackage{eucal}
\usepackage{fullpage}
\usepackage{scrtime}
\usepackage[colorlinks,citecolor=cyan]{hyperref}
\usepackage{datetime}
\usepackage{mathrsfs}

\shortdate
\renewcommand{\mathcal}{\mathscr}

\DeclareMathSymbol{A}{\mathalpha}{operators}{`A}%
\DeclareMathSymbol{B}{\mathalpha}{operators}{`B}%
\DeclareMathSymbol{C}{\mathalpha}{operators}{`C}%
\DeclareMathSymbol{D}{\mathalpha}{operators}{`D}%
\DeclareMathSymbol{E}{\mathalpha}{operators}{`E}%
\DeclareMathSymbol{F}{\mathalpha}{operators}{`F}%
\DeclareMathSymbol{G}{\mathalpha}{operators}{`G}%
\DeclareMathSymbol{H}{\mathalpha}{operators}{`H}%
\DeclareMathSymbol{I}{\mathalpha}{operators}{`I}%
\DeclareMathSymbol{J}{\mathalpha}{operators}{`J}%
\DeclareMathSymbol{K}{\mathalpha}{operators}{`K}%
\DeclareMathSymbol{L}{\mathalpha}{operators}{`L}%
\DeclareMathSymbol{M}{\mathalpha}{operators}{`M}%
\DeclareMathSymbol{N}{\mathalpha}{operators}{`N}%
\DeclareMathSymbol{O}{\mathalpha}{operators}{`O}%
\DeclareMathSymbol{P}{\mathalpha}{operators}{`P}%
\DeclareMathSymbol{Q}{\mathalpha}{operators}{`Q}%
\DeclareMathSymbol{R}{\mathalpha}{operators}{`R}%
\DeclareMathSymbol{S}{\mathalpha}{operators}{`S}%
\DeclareMathSymbol{T}{\mathalpha}{operators}{`T}%
\DeclareMathSymbol{U}{\mathalpha}{operators}{`U}%
\DeclareMathSymbol{V}{\mathalpha}{operators}{`V}%
\DeclareMathSymbol{W}{\mathalpha}{operators}{`W}%
\DeclareMathSymbol{X}{\mathalpha}{operators}{`X}%
\DeclareMathSymbol{Y}{\mathalpha}{operators}{`Y}%
\DeclareMathSymbol{Z}{\mathalpha}{operators}{`Z}%

\def\setminus{\mathchoice
    {\mathbin{\vrule height .72ex width 1.61ex depth -.38ex}}
    {\mathbin{\vrule height .72ex width 1.61ex depth -.38ex}}
    {\mathbin{\vrule height .50ex width 0.85ex depth -.28ex}}
    {\mathbin{\vrule height .20ex width 0.570ex depth -.24ex}}
}

\renewcommand{\leq}{\leqslant}
\renewcommand{\geq}{\geqslant}
 
\numberwithin{equation}{section}

\newcommand{\uple}[1]{\text{\boldmath${#1}$}}

\newcommand{\bfS}{\mathbf{S}}

\newcommand{\Cc}{\mathbf{C}}

\newcommand{\Aa}{\mathbf{A}}

\newcommand{\Zz}{\mathbf{Z}}

\newcommand{\Rr}{\mathbf{R}}

\newcommand{\Qq}{\mathbf{Q}}

\newcommand{\Ff}{\mathbf{F}}

\newcommand{\bQl}{\overline{\Qq}_{\ell}}

\newcommand{\proba}{\mathbf{P}}
\newcommand{\expect}{\mathbf{E}}

\newcommand{\normt}[1]{\|{#1}\|_{\mathrm{t}}}
\newcommand{\normf}[1]{\|{#1}\|_{\mathrm{tf}}}

\DeclareMathOperator{\frob}{Fr}

\newcommand{\injecte}{\hookrightarrow}

\DeclareMathOperator{\Spec}{Spec}

\DeclareMathOperator{\Kl}{\mathrm{Kl}}

\DeclareMathOperator{\Tr}{Tr}

\DeclareMathOperator{\ft}{FT}
\DeclareMathOperator{\cond}{\mathbf{c}}

\DeclareMathOperator{\dual}{D}

\newcommand{\eps}{\varepsilon}
\renewcommand{\rho}{\varrho}

\DeclareMathOperator{\SL}{SL}

\newcommand{\sheaf}[1]{\mathcal{{#1}}}

\DeclareMathSymbol{\gena}{\mathord}{letters}{"3C}
\DeclareMathSymbol{\genb}{\mathord}{letters}{"3E}

\theoremstyle{plain}
\newtheorem{theorem}{Theorem}[section]
\newtheorem*{theorem*}{Theorem}
\newtheorem{lemma}[theorem]{Lemma}
\newtheorem{corollary}[theorem]{Corollary}

\newtheorem{proposition}[theorem]{Proposition}

\theoremstyle{remark}

\theoremstyle{definition}

\newtheorem{definition}[theorem]{Definition}

\newtheorem{example}[theorem]{Example}
\newtheorem{remark}[theorem]{Remark}

\newcommand{\mcF}{\mathcal{F}}

\newcommand{\sfP}{\mathsf{P}}

\renewcommand{\geq}{\geqslant}
\renewcommand{\leq}{\leqslant}

\newcommand\sumsum{\mathop{\sum\sum}\limits}

\begin{document}

\title{Unmotivated ergodic averages}

\author{Emmanuel Kowalski}
\address{ETH Z\"urich -- D-MATH\\
  R\"amistrasse 101\\
  CH-8092 Z\"urich\\
  Switzerland} \email{kowalski@math.ethz.ch}

\begin{abstract}
  We consider weighted ergodic averages indexed by primes, where the
  weight depends on the prime, and is a ``trace function'' coming from
  algebraic geometry. We obtain extensions of classical results, in
  both $L^2$ and topological settings, and raise some further
  problems.
\end{abstract}

\subjclass[2010]{11T23, 11L05, 11N37, 11N75, 11F66, 14F20, 14D05}

\keywords{Riemann Hypothesis over finite fields, ergodic averages,
  ergodic theorems, maximal inequality, conductor of a sheaf, Fourier
  sheaf}

\maketitle 

\bigskip
\bigskip

\setcounter{tocdepth}{1}
\tableofcontents

\section{Introduction}

We consider a dynamical system $(X,\mu, f)$. Thus, $(X,\mu)$ is a
probability space and $f\colon X\to X$ is a measurable map such that
$f_*\mu=\mu$.

In this paper, motivated largely by simple curiosity (though see also
Remark~\ref{rm-motivation} for some arithmetical motivation), we
consider weighted ergodic averages of \emph{triangular} form\footnote{\
  In the sense of the ``triangular arrays'' of probability theory, e.g.,
  in the Central Limit Theorem.}, namely averages
\begin{equation}\label{eq-motiv}
  \frac{1}{p}\sum_{0\leq n<p} t_p(n)\ (\varphi\circ f^n),
\end{equation}
for some fixed function $\varphi\colon X\to \Cc$, where $p$ is a prime
and $t_p$ is a function on $\Zz$ (depending on $p$) ``of algebraic
origin''. Precisely, we are interested in the limit of such averages as
$p\to +\infty$ when the functions $t_p$ are \emph{trace function} modulo
$p$ or short linear combinations of such functions.


Since the general theory of trace functions (as amplified
by Fouvry, Kowalski and Michel in particular) is probably not
well-known to most readers, we present right away three basic
examples that will indicate the flavor of these averages.

\begin{example}\label{ex-trace-easy}
  (1) The function $t_p(n)$ which is the characteristic function of the
  set of squares modulo $p$ (quadratic residues) is essentially a linear
  combination
  $$
  \frac{1}{2}\Bigl(1+\Bigl(\frac{n}{p}\Bigr)\Bigr)
  $$
  of two trace functions. Thus the average~(\ref{eq-motiv}) is then
  essentially the ergodic average where $n$ is restricted to be a square
  modulo~$p$. The ``triangularity'' is very obvious here: as the prime
  $p$ changes, the set of quadratic residues modulo $p$ changes also.
  \par
  (2) Let $q\in \Zz[X]$ be a fixed monic polynomial. Then
  $t_p(n)=e(q(n)/p)$ is a trace function, where $e(z)=\exp(2i\pi z)$
  for any complex number $z$.
  \par
  (3) Define $t_p(n)=\Kl_2(n;p)$ where
  $$
  \Kl_2(n;p)=\frac{1}{\sqrt{p}}\sum_{1\leq x\leq p-1}
  e\Bigl(\frac{nx+\bar{x}}{p}\Bigr),
  $$
  where $\bar{x}$ is the inverse of $x$ modulo $p$. These are the
  classical \emph{Kloosterman sums}, which are of paramount importance
  in analytic number theory. The function $t_p$ is then also a trace
  function.
\end{example}

More generally, we will explain below the definition of two norms
$\normt{\cdot}\leq \normf{\cdot}$ on the space $\mathcal{C}(\Ff_p)$ of
complex-valued functions on~$\Ff_p=\Zz/p\Zz$, which we identify with the
interval $\{0,\ldots, p-1\}$. For $f\colon \Ff_p\to \Cc$, these norms
measure the complexity of a decomposition of~$f$ into sums of certain
trace functions. In the three examples above, we have
$\normf{t_p}\leq c$, where $c$ is independent of~$p$ (but depends on the
degree of~$q$ in the case of Example (2)), and similarly
$\normt{t_p}\leq c'$ for some constant $c'$, except in the case of
polynomials~$q$ of degree~$1$ in Example (2).

Then, exploiting the remarkable fundamental $L^2$ properties of trace
functions (which are very deep, as they rely on Deligne's most general
version of the Riemann Hypothesis over finite fields~\cite{deligne}),
we will prove rather easily the following result.

\begin{theorem}[$L^2$-ergodic theorems]\label{th-main}
  Let $(t_p)_p$ be a sequence of functions $t_p\colon\Ff_p\to\Cc$,
  indexed by an infinite subset~$\sfP$ of the primes. Let $(X,\mu,f)$ be
  a dynamical system and let
  $$
  \pi\colon L^1(X,\mu)\to L^1(X,\mu)
  $$
  be the projection given by the ergodic theorem
  \textup{(see~\cite[Th.\,2.30]{einsiedler-ward})}. Assume that there
  exists a constant $c\geq 0$ such that either
  \begin{enumerate}
  \item[(a)] We have $\normf{t_p}\leq c$ for $p\in\sfP$,
  \item[(b)] The system $(X,\mu,f)$ is weakly-mixing and
    $\normt{t_p}\leq c$ for $p\in\sfP$.
  \end{enumerate}
  \par
  Let $\varphi\in L^2(X,\mu)$.  Then the following results hold:
  \par
  \emph{(1)} We have
  $$
  \frac{1}{p}\sum_{0\leq n<p}t_p(n) \, \varphi\circ f^n-
  \Bigl(\frac{1}{p}\sum_{0\leq n<p}t_p(n)\Bigr)\pi(f)\to 0
  $$
  in $L^2(X,\mu)$ as $p\to +\infty$ along $\sfP$. Moreover, the
  convergence is uniform for~$\varphi$ in compact sets of $L^2(X,\mu)$.
  \par
  \emph{(2)} Suppose that 
  \begin{equation}\label{eq-sparse}
    \sum_{p\in\mathsf{P}}\frac{(\log p)^2}{p}<+\infty.
  \end{equation}
  
  Then for $\mu$-almost all $x$, we have
  $$
  \frac{1}{p}\sum_{0\leq n<p}t_p(n) \, \varphi(f^n(x))-
  \Bigl(\frac{1}{p}\sum_{0\leq n<p}t_p(n)\Bigr)\pi(f)(x)\to 0
  $$
  as $p\to +\infty$ along $\sfP$.
\end{theorem}

In addition, we consider the analogue of Sarnak's Möbius randomness
conjecture~\cite{sarnak} (one of the recent focus points at the
intersection of analytic number theory and ergodic theory) for our
weighted averages. We can prove a version of this conjecture for
certain specific families of trace functions, but since their definition
is non-trivial, we only state here some representative examples.

\begin{theorem}[Topological ergodic theorems]\label{th-sarnak}
  Let $X$ be a compact topological space and $f\colon X\to X$ a
  continuous map.

  Assume that the topological entropy of $f$ is zero.  Then for all
  continuous functions $\varphi\colon X\to \Cc$ and all $x\in X$, we
  have
  \begin{gather*}
    \lim_{p\to+\infty} \frac{1}{p}\sum_{0\leq n<p}
    \Kl_2(n;p)\varphi(f^n(x))=0,
    \\
    \lim_{p\to+\infty} \frac{1}{p}\sum_{0\leq n<p}
    \Bigl(\frac{n}{p}\Bigr)\varphi(f^n(x))=0.
  \end{gather*}
\end{theorem}

\begin{remark}
  (1) Sequences of the form $(\varphi(f^n(x)))_n$, where $f$ has
  topological entropy $0$ and $\varphi$ is continuous are called
  \emph{deterministic}.  Hence, the result shows that there is no
  deterministic sequence which can correlate non-trivially with an
  infinite sequence of Kloosterman sums, or Legendre symbols, modulo
  primes.

  (2) We will show that these families may be replaced by a fairly
  wide class of trace functions, but not all.

  (3) See~\cite{bowen} for Bowen's definition of topological entropy,
  which applies to uniformly continuous maps between metric spaces,
  and~\cite{akm} for the definition of Adler, Konheim and McAndrew which
  applies to arbitrary compact spaces. It is known that these are equal
  (when both are defined), see,
  e.g.,~\cite[Satz\,4.8]{einsiedler-schmidt}. Theorem~\ref{th-sarnak}
  seems likely to also hold for locally compact metric spaces and
  bounded uniformly continuous functions, but we haven't checked this
  (although it is a natural framework, e.g. for unipotent flows).

  (4) The special case of this theorem concerning Kloosterman sums was
  proved independently by El Abdalaoui, Shparlinski and
  Steiner~\cite[Th.\,2.8]{ess}.
\end{remark}

\begin{remark}\label{rm-motivation}
  From the arithmetic point of view, it is a crucial fact that
  \emph{there is no systematic rule to construct or constrain the
    sequences of trace functions that are used in the averages for each
    prime}.  We think that the sequence of Kloosterman sums or Legendre
  symbols are natural, but the only constraint that we impose in
  Theorem~\ref{th-main} is the boundedness of the trace norms of the
  functions (as in much previous work). We will see that the situation
  is very unclear when the system $(X,\mu,f)$ is not weakly-mixing.

  In general, the search for natural stronger conditions that ``bind'' a
  sequence $(t_p)$ of trace functions is, for the author, a very natural
  arithmetic motivation for the study of our weighted ergodic
  averages. In other words: is there a natural ``coherence'' condition
  for trace functions modulo primes that would naturally distinguish
  examples like Kloosterman sums or Legendre symbols?
\end{remark}

\subsection*{Outline of the paper}

We present some concrete ``incarnations'' of the results in
Section~\ref{sec-examples}. Then Section~\ref{sec-trace} gives the
definitions and basic background results concerning trace functions,
including defining the ``trace norms'' $\normt{\cdot}$ and
$\normf{\cdot}$.  Sections~\ref{sec-mean} and~\ref{sec-mean-general}
prove the mean ergodic theorem, and Section~\ref{sec-sarnak} discusses
the topological case. We then conclude with a discussion section
(including an easy maximal inequality in $L^2$), and with some further
questions that may be of interest in probing further the links between
these two subjects.

\subsection*{Notation} For basic references concerning ergodic
theory, we will refer to the books of Einsiedler and
Ward~\cite{einsiedler-ward} and of Einsiedler and
Schmidt~\cite{einsiedler-schmidt} (e.g., for topological entropy, which
is not discussed in~\cite{einsiedler-ward}).

We will summarize in Section~\ref{sec-trace} the key facts concerning
trace functions. More details and examples can be found for instance in
the surveys~\cite{pisa, aws} of Fouvry, Kowalski, Michel and Sawin.

We will say that an infinite set $\sfP$ of primes that
satisfies~(\ref{eq-sparse}) is \emph{sparse}.  In order that $\sfP$ be
sparse, it is enough that there exists~$\delta>0$ such that the counting
function
$$
\pi(x;\sfP)=\sum_{\substack{p\leq x\\p\in\sfP}}1
$$
satisfies
$$
\pi(x;\sfP)\ll \frac{x}{(\log x)^{3+\delta}}.
$$

\subsection*{Remark on the text}
  The first draft of these notes was written in 2018/2019. At that time,
  I put them aside: the absence of applications diminished the interest
  of the questions, and moreover the results did not seem strong enough
  (or the proofs conceptually interesting enough) to compensate this
  fault.

  I came back to the text in 2023, first because the appearance of the
  preprint~\cite{ess} of El Abdalaoui, Shparlinski and Steiner showed
  that at least a few other mathematicians did consider similar
  questions, and then because I decided to talk about this at least
  once, in the Number Theory Seminar of the University of Turku (where I
  was present to be the opponent in the PhD defense of
  O. Järvienemi). Although the defects discussed above still
  apply,\footnote{\ In addition to the fact that there might be lurking
    mistakes and imprecisions, and that there are significant
    redundancies in certain arguments.} there is (I think) one interesting
  outcome from working on this topic, namely the diophantine
  approximation result of Lemma~\ref{lm-single}, which was actually
  stated without proof in the 2019 draft.

\subsection*{Acknowledgements} Thanks to M. Einsiedler for discussions
about ergodic theory, and to L. Pierce for discussions concerning
maximal theorems.  Thanks to K. Matomäki for the invitation to be the
opponent of O. Järvienemi, which provided me with the occasion to revise
these notes, and Y. Bugeaud for remarks and references concerning
Lemma~\ref{lm-single}.

\section{Examples of results}\label{sec-examples}

Many of our results may be interpreted as leading to cancellation
properties for certain sums involving trace functions. These are often
of interest in analytic number theory, and we therefore state in this
section a few examples with concrete choices of trace functions and of
dynamical systems $(X,\mu,f)$. We also present examples which show that
some of the assumptions of Theorems~\ref{th-main} and~\ref{th-sarnak}
are needed.

\begin{example}\label{ex-1}
  We give first some examples related to continued fraction
  expansions.  Let $(]0,1[,\mu,f)$ be the continued fraction dynamical
  system (see~\cite[Ch. 3]{einsiedler-ward}), in other words
  $$
  \mu=\frac{1}{\log(2)}\frac{dx}{1+x},\quad\quad
  f(x)=\frac{1}{x}-\Bigl\lfloor\frac{1}{x}\Bigr\rfloor.
  $$
  This system is ergodic (loc. cit.) and~$f$ has positive entropy.

  For $x\in [0,1]$, let $(a_n(x))$ be the sequence of partial
  quotients in the continued fraction expansion of $x$. We have
  $a_{n+1}(x)=a_n(f(x))$.
  \par
  Maybe the simplest result that we can deduce from this work is that
  for a fixed integer $k\geq 0$, and for almost all $x$, we have
  $$
  \frac{1}{p} \Bigl| \Bigl\{ 1\leq n<p\,\mid\,
  \Bigl(\frac{n}{p}\Bigr)=1\text{ and } a_n(x)=k\Bigr\}\Bigr| \to
  \frac{1}{2\log 2} \log\Bigl(\frac{(k+1)^2}{k(k+2)}\Bigr)
  $$
  as $p\to +\infty$ along a sparse sequence, where $(n/p)$ is the
  Legendre symbol. This is one half of the density of occurence of
  $a_n(x)=k$, see~\cite[Cor. 3.8]{einsiedler-ward}.
  \par
  This result follows from Theorem~\ref{th-main}, (2) when we take
  $$
  t_p(n)=\frac{1}{2}\Bigl(1+\Bigl(\frac{n}{p}\Bigr)\Bigr),
  $$
  for $p$ odd, and $\varphi$ the characteristic function of
  $a_1(x)=k$, since $\varphi\circ f^n$ is the characteristic function
  of $a_n(x)=k$, and moreover we have $\normf{t_p}\ll 1$ and
  $$
  \frac{1}{p}\sum_{n\in\Ff_p} t_p(n)=\frac{1}{2}.
  $$
\end{example}

\begin{example}
  For $p$ prime, let $t_p$ be the Kloosterman sum function modulo~$p$
  (Example~\ref{ex-trace-easy}, (3)). We have $\normf{t_p}\ll 1$
  and
  $$
  \frac{1}{p}\sum_{n\in\Ff_p}t_p(n)=0.
  $$

  Define $X=\SL_2(\Zz)\backslash \SL_2(\Rr)$ and denote by $\mu$ the
  invariant probability measure on~$X$ (induced by a normalized Haar
  measure on $\SL_2(\Rr)$). Consider the dynamical system with
  $$
  f(g)=g\begin{pmatrix}
    2&0\\0&1/2
  \end{pmatrix}
  $$
  for~$g\in X$ (a part of the geodesic flow). It is known that
  $(X,\mu,f)$ is ergodic and that~$f$ has positive topological entropy.
  \par
  Let $\varphi\colon X \to \Cc$ be an $L^2$-function. Applying
  Theorem~\ref{th-main}, (2), we deduce that for almost all $z\in X$,
  we have
  $$
  \frac{1}{p}\sum_{1\leq n<p}\Kl_2(n;p)
  \varphi\Bigl(z\begin{pmatrix}
    2^n&0\\0&2^{-n}
  \end{pmatrix}\Bigr)\to 0
  $$
  as $p\to +\infty$ along a sparse subsequence.

  On the other hand, let
  $$
  \widetilde{f}(g)=g\begin{pmatrix} 1&1\\0&1
  \end{pmatrix}
  $$
  for~$g\in X$ (part of the horocycle flow). Then
  $(X,\mu,\widetilde{f})$ is ergodic and has zero entropy (note that~$X$
  is not compact, but~$\widetilde{f}$ is uniformly continuous, so
  Bowen's definition of entropy applies). Thus we have
  $$
  \frac{1}{p}\sum_{1\leq n<p}\Kl_2(n;p)
  \varphi\Bigl(z\begin{pmatrix}
    1&n\\0& 1
  \end{pmatrix}z\Bigr)\to 0
  $$
  for any bounded continuous function~$\varphi$ on~$X$ and any $z\in X$
  by Theorem~\ref{th-sarnak}.
\end{example}

\begin{example}\label{ex-fail}
  It is not surprising that pointwise convergence may fail in full
  generality, since this means considering arbitrary sequences $a_n$
  instead of $\varphi(f^n(x))$ (using the shift on $[-1,1]$ on the
  space of bounded sequences). As a simple example, consider again
  $t_p(n)=(n/p)$ (the Legendre symbol modulo $p$). 
  \par
  Let $(p_k)$ be an increasing sequence of primes with
  $p_{k+1}/p_k\to +\infty$; the set of primes thus defined is of
  course sparse. Define a sequence $a_n$ by
  $$
  a_n=\begin{cases}
    1\text{ if $n$ is a square modulo $p_{k+1}$}\\
    0\text{ if $n$ is not a  square modulo $p_{k+1}$},
  \end{cases}
  $$
  where $p_k\leq n<p_{k+1}$. Then
  $$
  \frac{1}{p_k}\sum_{0\leq n<p_k} t_{p_k}(n)a_n=
  1+O\Bigl(\frac{p_{k-1}}{p_k}\Bigr)\to 1.
  $$
  \par
  This example can, for instance, be embedded in the continued
  fraction setting, and can be adapted to pretty arbitrary sequences
  of trace functions.
\end{example}

\begin{example}
  Let $p$ be a prime and $t_p(n)=e(a_pn/p)$ for some
  $a_p\in\Ff_p$. These are trace functions, but we will show that
  Theorem~\ref{th-sarnak} does not hold with $\Kl_2(n;p)$ replaced by
  $t_p(n)$, at least if $(a_p)$ is chosen in a suitable manner.
  \par
  Pick $\theta\in\Rr/\Zz$ which is irrational. There exists $\delta>0$
  such that there are infinitely many approximations $a_p/p$ by rational
  numbers with prime denominators with
  \begin{equation}\label{eq-approx}
    \Bigl|\theta-\frac{a_p}{p}\Bigr|\leq \frac{1}{p^{1+\delta}}.
  \end{equation}
  \par
  Indeed, this was proved by Vinogradov for arbitrary $\delta<1/5$, and
  the best-known result by Matomäki~\cite{mato} applies for any
  $\delta<1/3$. The irrational translation $f(x)=x+\theta$ on $\Rr/\Zz$
  has entropy zero; pick the starting point $x=0$ and the continuous
  function $\varphi(\alpha)=e(\alpha)$ on $\Rr/\Zz$.  Then, for primes
  $p$ for which~(\ref{eq-approx}) holds, we get
  $$
  \frac{1}{p}\sum_{0\leq n<p}e\Bigl(-\frac{na_p}{p}\Bigr)e(n\theta)=
  \frac{1}{p}\frac{1-e(p(\theta-a_p/p))}{1-e(\theta-a_p/p)}\to 1
  $$
  as $p\to +\infty$ along this sequence.

  We note in passing that Theorem~\ref{th-main} does \emph{not} apply
  here (the system is not weakly-mixing, and the norms $\normf{t_p}$ are
  not bounded).

  (Also, we note that one could obtain easier examples using the fact
  that~(\ref{eq-approx}) holds with $\delta=1$ for almost all~$\theta\in
  [0,1]$, which goes back at least to Duffin and
  Schaeffer~\cite{duffin-schaeffer}.) 
\end{example}

\begin{example}\label{ex-traces}
  Here are some additional standard examples of functions on $\Zz$ that
  arise as trace functions modulo $p$ with bounded conductor, and which
  moreover are ``geometrically irreducible'' (an important property
  which means essentially that their mean-square average modulo~$p$ is
  close to~$1$). More examples are found, e.g., in~\cite{pisa}. This
  should give an idea of the variety of ergodic averages that we are
  considering.

  (1) For any $a$ modulo~$p$, the additive character $n\mapsto e(an/p)$
  is a trace function of a so-called Artin-Schreier sheaf; it has
  conductor uniformly bounded.

  (2) For any non-trivial multiplicative character $\chi$ modulo~$p$,
  extended by~$0$ to $\Zz/p\Zz$, the corresponding Dirichlet character
  is a trace function of a so-called Kummer sheaf; it has conductor
  uniformly bounded.

  (3) More generally, let $f\in\Zz[X]$ be a non-constant
  polynomials. The functions $n\mapsto e(f(n)/p)$ and
  $n\mapsto \chi(f(n))$ are trace functions with conductor bounded in
  terms of the degree of~$f$ only. Similarly if~$f$ is a non-constant
  rational function, with the trace function having value~$0$ at poles
  of~$f$, and with conductor depending on the degrees of the numerator
  and denominators of~$f$.

  (4) If~$t_p$ is a geometrically irreducible trace function modulo~$p$,
  and is not proportional to an additive character, then its normalized
  Fourier transform
  $$
  \widehat{t}_p(n)=\frac{1}{\sqrt{p}}\sum_{0\leq m<p}
  e\Bigl(\frac{mn}{p}\Bigr)t_p(m)
  $$
  is also a trace function with conductor bounded only in terms of that
  of~$t_p$ (see~\cite[Prop.\,8.2]{fkm1}).

  So for instance, the fact that the Kloosterman sums used above (see
  Example~\ref{ex-trace-easy}), namely
  $$
  t_p(n)=\frac{1}{\sqrt{p}}\frac{1}{\sqrt{p}}\sum_{1\leq x\leq p-1}
  e\Bigl(\frac{nx+\bar{x}}{p}\Bigr),
  $$
  define a geometrically irreducible  trace function modulo~$p$ with
  bounded conductor follows from this principle applied to the trace
  function $n\mapsto e(\bar{n}/p)$ (extended by~$0$ for~$n=0$), which is
  a special case of Example (3).

\end{example}

As a final remark, we emphasize that trace functions behave in many ways
like random functions (e.g., they often have Gowers norms that are as
small as those of random functions, as shown by Fouvry, Kowalski and
Michel in~\cite{fkm-gowers}), and one can think of them in these terms
in a first reading.

\section{Properties of trace functions}
\label{sec-trace}

We summarize here the properties of trace functions that we will
use. These are essentially related to the Fourier transform (which was
already mentioned in Example~\ref{ex-traces}, (4), as an operation
preserving trace functions).

First, we fix throughout the paper a prime number $\ell$, and impose
that all other prime numbers we consider below are different from
$\ell$ (one can take~$\ell=2$ and only consider odd primes). We fix an
isomorphism $\iota\colon \bQl\to \Cc$. We first clarify our
terminology and conventions concerning sheaves:

\begin{definition}[Sheaves and uniform sheaves]\label{def-sheaf}
  Let $p\not=\ell$ be a prime.

  (1) A \emph{sheaf} $\sheaf{F}$ modulo $p$ is a middle extension
  $\bQl$-sheaf on $\Aa^1_{\Ff_p}$, pure of weight $0$. A \emph{Fourier
    sheaf} modulo $p$ is a sheaf modulo $p$ that is of Fourier type in
  the sense of Katz~\cite[7.3.4]{katz-esde}, in other words, none of
  its geometrically irreducible components is geometrically isomorphic
  to an Artin-Schreier sheaf.
  \par
  (2) The \emph{trace function} $t_{\sheaf{F}}$ of a sheaf $\sheaf{F}$
  modulo $p$ is the complex-valued function on $\Zz$ defined by
  $$
  t_{\sheaf{F}}(x)=\iota(\Tr(\frob_{x,\Ff_p}|\sheaf{F}_{\bar{x}}))
  $$
  where $\frob_{x,\Ff_p}$ is the Frobenius at $x\in\Ff_p$, and
  $\bar{x}$ is a geometric point above $x$.
  \par
  (3) Let $\sheaf{F}$ be a sheaf modulo $p$ and let $k\geq 0$ be an
  integer. We say that $\sheaf{F}$ is \emph{$k$-uniform} if no
  geometrically irreducible component of $\sheaf{F}$ is geometrically
  isomorphic to a sheaf of the type $\sheaf{L}_{\psi(P)}$ where $\psi$
  is a non-trivial additive character of $\Ff_p$ and $P\in\Ff_p[X]$ is
  a polynomial of degree $\leq k$.
  \par
  (4) We say that $\sheaf{F}$ is almost $k$-uniform with average $\mu$
  if $\sheaf{F}\simeq \mu\bQl\oplus \sheaf{G}$ where $\sheaf{G}$ is
  $k$-uniform.
\end{definition}

Note that speaking of geometrically irreducible components of a
sheaf~$\mcF$ modulo~$p$ is legitimate, since such sheaves (being pure
of some weight) are geometrically semisimple by work of Deligne.

\begin{example}
  To say that $\sheaf{F}$ is $0$-uniform (resp. $1$-uniform) means that
  $\sheaf{F}$ has no trivial geometrically irreducible component
  (resp. is of Fourier type in the sense of
  Katz~\cite[7.3.5]{katz-esde}).
\end{example}

Let~$\sheaf{F}$ be a sheaf modulo~$p$. Fouvry, Kowalski and Michel
defined its \emph{conductor} $\cond(\sheaf{F})$
in~\cite[Def. 1.13]{fkm1}; it is a positive integer which vanishes if
and only if~$\sheaf{F}$ is zero.

The conductor measures quantitatively the complexity of a sheaf in
many estimates.. One essential property is a bound on the size of the
trace function: for any sheaf $\sheaf{F}$ modulo $p$, and any
$x\in\Zz$, we have
\begin{equation}\label{eq-linfty}
  |t_{\sheaf{F}}(x)|\leq \cond(\sheaf{F}).
\end{equation}

Using the conductor, we define the trace norms as follows.

\begin{definition}[Trace norms]\label{def-trace-norms}
  Let~$p$ be a prime different from~$\ell$. Let~$\mathcal{C}(\Ff_p)$
  denote the vector space of $\Cc$-valued functions on~$\Ff_p$.

  For~$f\in\mathcal{C}(\Ff_p)$, we define
  $$
  \normt{f}= \inf\Bigl\{\sum_i\cond(\sheaf{F}_i)|a_i|\,\mid\, f=\sum_{i}
  a_it_{\sheaf{F}_i},\ \sheaf{F}_i\text{ geometrically
    irreducible}\Bigr\},
  $$
  and
  $$
  \normf{f}= \inf\Bigl\{\sum_i\cond(\sheaf{F}_i)|a_i|\,\mid\, f=\sum_{i}
  a_it_{\sheaf{F}_i},\ \sheaf{F}_i\text{ geometrically irreducible
    Fourier}\Bigr\}.
  $$
  \par
  In both cases, the infimum runs over decompositions of~$f$ in linear
  combinations of trace functions of sheaves of the indicated type.
\end{definition}

It is straightforward that both of these are norms, and clear
that~$\normt{f}\leq\normf{f}$.

\begin{remark}
  Although we mentioned that trace functions can be thought of as
  ``random'' functions, one should note that for most simple models of
  random functions $f\colon \Ff_p\to\Cc$ (e.g., taking all $f(n)$ to be
  independent and uniform over the unit disc), the norm $\normt{f}$ will
  in fact be very large, as explained in a paper of Fouvry, Kowalski and
  Michel (see~\cite[Th. 5.1]{counting-sheaves}).
\end{remark}
\par
We now state some of the fundamental analytic properties of trace
functions, starting with the general form of the ``completion method''
for short sums of trace functions.

\begin{proposition}
  [Completion method]
  \label{pr-completion}
Let $\sheaf{F}$ be a Fourier sheaf
modulo $p$ and $t\colon \Zz\to \Cc$ its trace function. For any
interval $I$ in $\Zz$ of length $\leq p$, we have
$$
\sum_{n\in I}t(n)\ll \sqrt{p}(\log p)
$$
where the implied constant depends only on the conductor of
$\sheaf{F}$.
\end{proposition}

See~\cite[\S 1.1, \S 2.2]{short-sums} for the argument, which is
straightforward, granted the very deep fact (a case of Deligne's
Riemann Hypothesis in its strongest form) that the normalized discrete
Fourier transform of $t$ is the trace function of a sheaf
$\ft(\sheaf{F})$, which is also a middle-extension of weight $0$, with
conductor $\leq 10 \cond(\sheaf{F})^2$ (this last important estimate
is proved by Fouvry, Kowalski and Michel in~\cite[Prop. 8.2]{fkm1}).

\begin{proposition}
\label{pr-additive-convol}
Let $\sheaf{F}$ and $\sheaf{G}$ be middle-extension $\ell$-adic
sheaves of weight $0$ modulo $p$. 
\par
\emph{(1)} The additive middle convolution $\sheaf{F}*_! \sheaf{G}$ is
a middle-extension $\ell$-adic sheaf of weights $\leq 0$, and it has
conductor bounded in terms of the conductors of $\sheaf{F}$ and
$\sheaf{G}$.
\par
\emph{(2)} Suppose that $\sheaf{F}$ is a Fourier sheaf. The additive
middle convolution $\sheaf{F}*_!  \dual(\sheaf{F})$ contains no
Artin-Schreier sheaf as geometrically irreducible component.
\end{proposition}

\begin{proof}
  For (1), the first assertion follows from the definition of the
  middle convolution and from Deligne's Riemann Hypothesis. To
  estimate the conductor, it is simplest here to appply the Fourier
  transform, which is an exact functor transforming
  middle-convolution into tensor product, so that
  $$
  \sheaf{F}*_! \sheaf{G}=\overline{\ft}(\ft(\sheaf{F})\otimes
  \ft(\sheaf{G})).
  $$
  We can then apply the estimate~\cite[Prop.\,8.2]{fkm1} for the
  conductor of a Fourier transform.
  \par
  For (2), applying the Fourier transform, a hypothetical injection
  $\sheaf{L}_{\psi(ax)}\injecte \sheaf{F} *_! \dual(\sheaf{F})$ would
  imply the existence of an injection
  $$
  \delta_a \injecte \ft_{\psi}(\sheaf{F})\otimes
  \dual(\ft_{\psi}(\sheaf{F}))
  $$
  of a punctual skyscraper sheaf into
  $\ft_{\psi}(\sheaf{F})\otimes \dual(\ft_{\psi}(\sheaf{F}))$. However,
  since both $\ft_{\psi}(\sheaf{F})$ and its dual are middle-extension
  sheaves when $\sheaf{F}$ is a middle-extension, their tensor product
  has no punctual part.
\end{proof}

The following definition will be convenient in some places.

\begin{definition}\label{def-af}
  A family $(\sheaf{F}_p)_p$ of sheaves modulo $p$ indexed by (a
  subset of) the primes $\not=\ell$ is an \emph{almost Fourier family}
  if the conductor of $\sheaf{F}_p$ is bounded independently of $p$,
  and if there exists an integer $r\geq 0$ such that
  $\sheaf{F}_p=r\bQl\oplus \widetilde{\sheaf{F}}_p$ for all $p$, where
  $\widetilde{\sheaf{F}}_p$ is a Fourier sheaf modulo $p$. We say that
  $r$ is the mean of the family.
\end{definition}

For an almost Fourier family, the trace functions $t_p$ of
$\sheaf{F}_p$ satisfy
$$
t_p(x)=r+\widetilde{t}_p(x)
$$
where $\widetilde{t}_p$ is the trace function of
$\widetilde{\sheaf{F}}_p$.

The following proposition will be only be used for polynomials $P$ of
degree~$1$, but since it is of independent interest, we state and
prove it in general (see~\cite[Th.\,2.7]{ess} for a special case).

\begin{proposition}\label{pr-weyl}
  Let $k\geq 1$ be an integer and define $\gamma_k=2^{-k}$. Let $p$ be
  a prime and let $\sheaf{F}$ be a $k$-uniform $\ell$-adic sheaf
  modulo $p$ with trace function $t(n)$. Let $P\in\Rr[X]$ be a
  polynomial of degree $\leq k$. Let $I$ be an interval in $\Zz$ of
  length $|I|\geq 1$. We have
  \begin{equation}\label{eq-weyl}
    \sum_{n\in I} t(n)e(P(n)) \ll \cond(\sheaf{F})^2
    \Bigl(|I|^{1-2\gamma_k}p^{\gamma_k}+|I|p^{-\gamma_k}\Bigr)
    (\log p)^{2\gamma_k}
  \end{equation}
  where the implied constant is absolute.
\end{proposition}

For $k\geq 2$, the proof will use the following lemma; readers only
interested in main results of this paper may skip this in a first
reading.

\begin{lemma}\label{lm-fourier}
  Let $k\geq 1$ be an integer and $p$ a prime.  Let $\sheaf{F}$ be a
  geometrically isotypic $k$-uniform $\ell$-adic sheaf modulo $p$ for
  some integer $k\geq 1$. Let $h\in\Ff_p$ be such that the set of
  singularities of $[+h]^*\sheaf{F}$ and $\dual(\sheaf{F})$ are
  disjoint.  If $p>k$ and $\cond(\sheaf{F})<p$, and if $h\not=0$, then
  $[+h]^*\sheaf{F}\otimes\dual(\sheaf{F})$ is a $(k-1)$-uniform
  $\ell$-adic sheaf modulo $p$ with
  conductor~$\ll \cond(\sheaf{F})^2$.
\end{lemma}

\begin{proof}
  This is implicit in the work of Fouvry, Kowalski and Michel
  in~\cite[\S 5]{fkm-gowers}. Precisely, under the assumption on $h$,
  the tensor product $[+h]^*\sheaf{F}\otimes\dual(\sheaf{F})$ is an
  $\ell$-adic sheaf modulo $p$ (the key point is that it is a
  middle-extension, see~\cite[Lemma 2.2]{fkm-gowers}).  If the
  conclusion does not hold, we deduce from the definition of
  $(k-1)$-uniform sheaf that there exists a polynomial $P$ of degree
  $\leq k-1$ such that
  $$
  \sheaf{F}\simeq [+h]^*\sheaf{F}\otimes\sheaf{L}_{\psi(P)}
  $$
  (see~\cite[Lemma 5.3 (2)]{fkm-gowers}). From this, we see first that
  $\cond(\sheaf{F})\geq p$, if $\sheaf{F}$ is not lisse on
  $\Aa^1_{\Ff_p}$ (because the orbit of a singularity under
  $x\mapsto x+h$ is contained in the set of singularities, so there
  are at least $p$ of them, each of which contributes at least $1$ to
  the sum of drops of $\sheaf{F}$). Otherwise, since $p>k$,
  by~\cite[Lemma 5.4 (2)]{fkm-gowers}, it follows that either
  $\cond(\sheaf{F})\geq p$ (because of the contribution of the Swan
  conductor at $\infty$) or $\sheaf{F}$ is isomorphic to
  $\sheaf{L}_{\psi(Q)}$ for some polynomial of degree $\leq k$. The
  lemma follows, by contraposition.
\end{proof}

\begin{proof}
  We first consider the case $|I|\leq p$. We then need to show that
  \begin{equation}\label{eq-target}
    \sum_{n\in I} t(n)e(P(n)) \ll \cond(\sheaf{F})^2
    |I|^{1-2\gamma_k}p^{\gamma_k}(\log p)^{2\gamma_k},
  \end{equation}
  and we may assume (by additive change of variable) that $I$ is
  contained in $\{0,\ldots,p-1\}$.
  \par
  We assume (as we may) that~$P(0)=0$. If we decompose the arithmetic
  semisimplification of $\sheaf{F}$ in arithmetically irreducible
  components, say $\sheaf{F}_i$, then one of the following is true
  (see~\cite[Lemma 5.3]{fkm-gowers}):
  \par
  (1) For some $n\geq 2$, the sheaf $\sheaf{F}_i$ is induced from some
  irreducible sheaf on $\Spec(\Ff_{p^n})$ by pushforward along the map
  $\Spec(\Ff_{p^n})\to \Spec(\Ff_p)$; in this case, the trace function
  $t_i$ of $\sheaf{F}_i$ is identically $0$ (see~\cite[Lemma
  5.3]{fkm-gowers} or~\cite[Proof of Prop. 8.3]{fkm1}), so that the
  estimate~(\ref{eq-weyl}) is trivial.
  \par
  (2) The sheaf $\sheaf{F}_i$ is geometrically isotypic.
  \par
  Since the estimate~(\ref{eq-weyl}) is linear in $\sheaf{F}$, we see
  that we may reduce the proof to the case where $\sheaf{F}$ is
  geometrically isotypic.
  \par
  We now proceed by induction on $k$. The key tool is Weyl
  differencing.  Assume first that $k=1$ and that $P(n)=\theta n$
  (here we do not need to assume that $\sheaf{F}$ is isotypic). By
  discrete Fourier inversion, we obtain
  $$
  \sum_{n\in I}t(n)e(\theta n)= \sum_{0\leq h<p} \widehat{t}(h)
  \alpha_p(h,\theta)
  $$
  where
  $$
  \alpha_p(h,\theta)=\frac{1}{\sqrt{p}}\sum_{n\in I}
  e\Bigl(n\Bigl(\frac{h}{p}+\theta\Bigr)\Bigr), \quad\quad
  \widehat{t}(h)=\frac{1}{\sqrt{p}}\sum_{0\leq
    n<p}t(n)e\Bigl(\frac{nh}{p}\Bigr).
  $$
  \par
  Since $\sheaf{F}$ is $1$-uniform, it is a Fourier sheaf, and we have
  $|\widehat{t}(h)|\leq \cond(\ft(\sheaf{F}))\ll \cond(\sheaf{F})^2$
  (by~\cite[Prop.\,8.2]{fkm1}). On the other hand, by summing the
  geometric sum, we have
  $$
  |\alpha_p(h,\theta)|\leq
  \min\Bigl(\frac{|I|}{\sqrt{p}},\frac{1}{\sqrt{p}}
  \frac{1}{\|\tfrac{h}{p}+\theta\|}\Bigr),
  $$
  where $\|\cdot\|$ on the right-hand side is the distance to the
  nearest integer. We use the first bound for that value $h_0$ of $h$
  where $|h_0/p+\theta|\leq 1/p$, and the other values of
  $\alpha_p(h,\theta)$ are then bounded by
  $$
  \frac{\sqrt{p}}{2},\quad\cdots,\quad \frac{\sqrt{p}}{p},
  $$
  so that
  $$
  \sum_{0\leq h<p}|\alpha_p(h,\theta)|\ll \sqrt{p}(\log p),
  $$
  with an absolute implied constant.  Combining these results we
  obtain
  $$
  \Bigl| \sum_{0\leq h<p} \widehat{t}(h) \alpha_p(h,\theta) \Bigr|
  \ll \cond(\sheaf{F})\sqrt{p}\log p,
  $$
  with an absolute implied constant, which implies the
  bound~(\ref{eq-target}) for $k=1$.

  Now assume that $\deg(P)=k\geq 2$ and that the proposition is true
  for polynomials of degree $k-1$; assume (as we saw that we may) that
  $\sheaf{F}$ is geometrically isotypic. We write
  \begin{align*}
    \Bigl|\sum_{n\in I}t(n)e(P(n))\Bigr|^2
    &=
      \sum_{n,m\in I}t(n)\overline{t(m)}e(P(n)-P(m))
    \\
    &=\sum_{h\in I-I}
      \sum_{m\in I_h}t(m+h)\overline{t(m)}e(P(m+h)-P(m))\\
    &=\sum_{h}
      \sum_{m\in I_h}t(m+h)\overline{t(m)}e(Q_h(m))
  \end{align*}
  where $Q_h=P(X+h)-P(X)$ is a polynomial of degree $\leq k-1$ and
  $I_h$ is an interval, depending on $h$, of length $|I_h|\leq |I|$.

  For $h\in I-I$ such that the set of singularities of
  $[+h]^*\sheaf{F}$ and $\dual(\sheaf{F})$ are not disjoint, we use
  the trivial bound
  $$
  \Bigl| \sum_{m\in I_h}t(m+h)\overline{t(m)}e(Q_h(m))\Bigr| \leq
  \cond(\sheaf{F})^2|I|.
  $$
  Note that there are at most $n^2$ values of $h$ with this property,
  where $n\leq \cond(\sheaf{F})$ is the number of singularities
  of~$\sheaf{F}$.
  \par
  Now suppose that the set of singularities of $[+h]^*\sheaf{F}$ and
  $\dual(\sheaf{F})$ are disjoint. The function
  $$
  m\mapsto t(m+h)\overline{t(m)}
  $$
  is the trace function of the sheaf
  $[+h]^*\sheaf{F}\otimes\dual(\sheaf{F})$, which is $(k-1)$-uniform
  by Lemma~\ref{lm-fourier}. Hence, by induction, we have
  $$
  \sum_{m\in I_h}t(m+h)\overline{t(m)}e(Q_h(m)) \ll
  \cond(\sheaf{F})^2|I_h|^{1-2\gamma_{k-1}}p^{\gamma_{k-1}}(\log
  p)^{2\gamma_{k-1}}
  $$
  where the implied constant is absolute.  Finally, gathering the
  estimates together, since $|I-I|\leq 2|I|$ and $|I_h|\leq |I|$, we
  obtain
  $$
  \Bigl|\sum_{n\in I}t(n)e(P(n))\Bigr|^2\ll \cond(\sheaf{F})^4|I|+
  \cond(\sheaf{F})^2|I|^{2-2\gamma_{k-1}}p^{\gamma_{k-1}}(\log
  p)^{2\gamma_{k-1}},
  $$
  and~(\ref{eq-target}) follows for degree~$k$ by taking the square
  root since~$\gamma_{k-1}=2\gamma_k$.

  We now assume that~$|I|>p$. We can decompose the interval $I$ into
  $\lfloor |I|/p\rfloor$ intervals of length $p$ and one remaining
  interval~$J$ of length $|J|\leq p$. Using shifts, each of these sums
  is of the type above for a shifted sheaf, with the same conductor,
  and an interval of length~$\leq p$. The previous case therefore
  implies
  $$
  \sum_{n\in I} t(n)e(P(n)) \ll \cond(\sheaf{F})^2 \frac{|I|}{p}\times
  p^{1-2\gamma_k}p^{\gamma_k}(\log p)^{2\gamma_k}
  $$
  (since the implied constant is independent of the coefficients of
  $P$), and this is of the desired shape for~$|I|>p$.
\end{proof}

\begin{corollary}\label{cor-ft}
  Let $\sheaf{F}$ be a Fourier sheaf modulo $p$ with trace function
  $t_p$, and $\theta\in\Rr/\Zz$. We have
  $$
  \sum_{0\leq n<p}t_p(n)e(-\theta n)\ll \sqrt{p}\log p,
  $$
  where the implied constant depends only on the conductor of
  $\sheaf{F}$.
\end{corollary}

\begin{remark}
  (1) The estimate of Proposition~\ref{pr-weyl} cannot be improved
  without some additional assumption, since $t(n)=e(n^k/p)$ is the
  trace function of a sheaf that is $(k-1)$-uniform but not
  $k$-uniform.
  \par
  (2) Estimates similar to that of Proposition~\ref{pr-weyl} have been
  proved by a number of authors when $t(n)=\chi(n)$ is a
  multiplicative character, beginning with Enflo~\cite{enflo}; more
  recent works include those of Chang~\cite{chang}, Heath-Brown and
  Pierce~\cite{hb-p} and Pierce~\cite{pierce}. In that special case,
  rather stronger results hold; as far as the size of $I$ is
  concerned, they are comparable to the Burgess bound for short
  character sums, i.e., non-trivial provided that $I$ is a bit larger
  than $p^{1/4}$.
  \par
  (3) If~$\theta=a/p$ for some integer~$a$, then the estimate of
  Corollary~\ref{cor-ft} holds without the factor~$\log p$, by the
  existence of Deligne's Fourier transform. It would be interesting to
  know if this factor is really needed in general.
\end{remark}

\section{The mean ergodic theorem in the Fourier case}
\label{sec-mean}

This section considers the mean ergodic theorem in $L^2$. As can be
expected from the good $L^2$ properties of trace functions, a very
satisfactory theory exists, and it is reasonably easy to
derive. Roughly speaking, we will see that non-trivial interactions
arise only from the Artin-Schreier components (on the side of trace
functions) and from the Kronecker factor (on the dynamical side). So
if either the Artin-Schreier component or the Kronecker factor is
trivial (the latter means that the dynamical system is weakly-mixing),
then the statements are particularly clear.

We fix a measurable dynamical system $(X,\mu,f)$. We denote by
$$
u_f\colon L^2(X,\mu)\to L^2(X,\mu)
$$
the associated unitary operator, defined by $u_f(\varphi)=\varphi\circ
f$ for all $\varphi$.

We also fix a family $(\sheaf{F}_p)_p$ of sheaves modulo $p$ with
bounded conductor, indexed by an infinite set of primes~$\sfP$. We
denote by $t_p$ the trace function of $\sheaf{F}_p$, viewed as a
function on~$\Zz$.  Finally, we denote by
$$
v_p=\frac{1}{p}\sum_{0\leq n<p} t_p(n)\ u_f^n
$$
the ergodic averaging operator with weight $t_p$ acting on
$L^2(X,\mu)$.

\begin{proposition}\label{pr-l2}
  Suppose that $\sheaf{F}_p$ is a Fourier sheaf for all $p$.  The
  endomorphisms $(v_p)_p$
  of $L^2(X,\mu)$ converge to $0$ as $p\to +\infty$ with respect to
  the operator norm.
  In fact, we have
  \begin{equation}\label{eq-norm}
    \|v_p\|\ll p^{-1/2}(\log p),
  \end{equation}
  where the implied constant depends only on $\cond(\sheaf{F}_p)$.
\end{proposition}

Although the proof may seem rather trivial, it relies on the Riemann
Hypothesis over finite fields.

\begin{proof}
  Let $\varphi\in L^2(X,\mu)$ have norm $1$.  Let $\nu$ be the
  spectral measure of the unitary operator~$u_f$ relative to the unit
  vector~$\varphi$, i.e., the Borel probability measure on $\Rr/\Zz$
  such that
  $$
\int_{\Rr/\Zz}\rho(e(\theta))d\nu(\theta)=\langle
\rho(u_f)\varphi|\varphi\rangle
$$
for any continuous function $\rho$ on $\bfS^1$ (see,
e.g.,~\cite[Déf.\,4,\,p.\,268]{TS345}). We obtain in particular
$$
\|v_p(\varphi)\|^2= \int_0^1 \Bigl|\frac{1}{p}\sum_{0\leq
  n<p}t_p(n)e(n\theta)\Bigr|^2d\nu(\theta).
$$
Applying Corollary~\ref{cor-ft}, we get
$$
\|v_p(\varphi)\|^2\ll \frac{(\log p)^2}{p}
$$
where the implied constant depends only on the conductor of
$\sheaf{F}_p$. This concludes the proof.
\end{proof}

This implies the first part of Theorem~\ref{th-main}, in the case of
Fourier sheaves (with uniform convergence over bounded sets), because
$$
\frac{1}{p}\sum_{0\leq n<p}t_p(n)\ll \frac{1}{\sqrt{p}}\to 0
$$
in that case.

For arbitrary functions $t_p\colon \Ff_p\to \Cc$, provided they satisfy
Assumption~(a), namely $\normf{t_p}\ll 1$, we can represent $t_p$ as a
finite combination (with coefficients bounded in $\ell_1$) of trace
functions of Fourier sheaves, and obtain the same result by linearity.

Moreover, this also implies the second part, still in the case of
Fourier sheaves, by a standard trick: if $p$ ranges over a sparse set of
primes~$\sfP$, then for any fixed $\varphi\in L^2(X,\mu)$, the series
$$
\sum_{p} \|v_p(\varphi)\|^2
$$
converges (by~(\ref{eq-norm}) and the definition by sparseness), and
this implies that the function
$$
x\mapsto \sum_p|v_p(\varphi)(x)|^2
$$
is finite almost surely, hence that $v_p(\varphi)(x)$ converges to~$0$
for almost all~$x$. Once again, this gives the second part of
Theorem~\ref{th-main} under Assumption~(a) by linearity.

\begin{remark}
  For the sake of variety, here is an argument which provides a proof of
  the weaker result
  $$
  \|v_p\|\ll p^{-1/4}(\log p)^{1/2},
  $$
  without using the spectral theorem.  Let $\varphi\in L^2(X,\mu)$
  and
  $$
  \psi_p=v_p(\varphi)=\frac{1}{p}\sum_{0\leq n<p} t_p(n)\
  u_f^n(\varphi).
  $$
  We compute
  \begin{align*}
    \|\psi_p\|^2 
    &= \frac{1}{p^2}\sum_{\substack{0\leq n<p\\0\leq m<p}}
    t_p(n)\overline{t_p(m)}
    \langle u_f^n(\varphi)|u_f^m(\varphi)\rangle\\
    &=\frac{1}{p^2} 
      \sum_{|h|<p} \langle u_f^h(\varphi)|\varphi\rangle
      \sum_{\substack{0\leq n,m<p\\ n-m=h}}
    t_p(n)\overline{t_p(m)}.
  \end{align*}
  \par
  The contribution coming from $h=0$ is
  $$
  \frac{\|\varphi\|^2}{p^2}\sum_{x\in\Ff_p}|t_p(x)|^2\leq
  \cond(\sheaf{F}_p)^2\|\varphi\|^2p^{-1}
  $$
  by~(\ref{eq-linfty}).  Now fix $h$ with $1\leq |h|<p$. The
  corresponding summand is
  $\langle u_f^h(\varphi)|\varphi\rangle \sigma_h$, where
  $$
  \sigma_h= \sum_{\substack{0\leq n,m<p\\ n-m=h}} 
  t_p(n)\overline{t_p(m)} 
  = \sum_{\max(0,h)\leq  n<\min(p,p+h)} 
  t_p(n)\overline{t_p(n-h)}.
  $$
  By completion and by the properties of the additive convolution of
  trace functions of Fourier sheaves (see
  Proposition~\ref{pr-completion} and
  Proposition~\ref{pr-additive-convol}), we have
  $$
  \sigma_h\ll \sqrt{p}(\log p)
  $$
  for all $h\not=0$, where the implied constant depends only on
  $\cond(\sheaf{F}_p)$. Therefore we derive
  $$
  \|\psi_p\|^2 \ll \|\varphi\|^2p^{-1}+ \|\varphi\|^2 p^{-1/2}(\log p)
  $$
  where the implied constant depends only on $\cond(\sheaf{F}_p)$. This
  gives the result.
\end{remark}

We can immediately extend the mean-ergodic theorem for Fourier sheaves
to $L^r$ when $1\leq r\leq 2$. For $r>2$, see Section~\ref{sec-lr}.

\begin{corollary}\label{cor-lr12}
  Suppose that $\sheaf{F}_p$ is a Fourier sheaf for all $p$.  Let
  $r\in [1,2]$. The endomorphisms
  $$
  \widetilde{v}_p=\frac{1}{p}\sum_{0\leq n<p} t_p(n)\ u_f^n
  $$
  of $L^{r}(X,\mu)$ converge to $0$ as $p\to +\infty$ in the norm
  topology.
\end{corollary}

\begin{proof}
  Suppose first that $\varphi$ is bounded. Since $r\leq 2$ an $\mu$ is
  a probability measure, we have
  $$
  \|\widetilde{v}_p(\varphi)\|_{r}^r= \int_{X}\Bigl|
  \frac{1}{p}\sum_{0\leq n<p} t_p(n)\, \varphi(f^n(x)) \Bigr|^rd\mu(x)
  \leq \Bigl(\int_{X}\Bigl| \frac{1}{p}\sum_{0\leq n<p} t_p(n)\,
  \varphi(f^n(x)) \Bigr|^2d\mu(x)\Bigr)^{r/2},
  $$
  hence $\|\widetilde{v}_p\|\leq \|v_p\|$, which tends to $0$.
\end{proof}

Recall that we denote by $\pi$ the ergodic projection
$L^1(X,\mu)\to L^1(X,\mu)$. It restricts to the orthogonal projection on
the $1$-eigenspace of $L^2(X,\mu)$. The standard mean-ergodic theorem in
$L^2$ implies that
$$
\frac{1}{p}\sum_{0\leq n<p} u_f^n\to \pi
$$
in the space of endomorphisms of $L^2(X,\mu)$ with the topology of
pointwise convergence (see, e.g.,~\cite[Th.\,2.21]{einsiedler-ward}).

Recall further that almost Fourier
families are defined in Definition~\ref{def-af}.

\begin{corollary}
  Assume that the family $(\sheaf{F}_p)$ is almost Fourier with mean
  $r\geq 0$. Then the sequence of endomorphisms $(v_p)$ of
  $L^2(X,\mu)$ converges to $r\pi$ as $p\to +\infty$ with respect to
  the topology of uniform convergence on compact subsets of
  $L^2(X,\mu)$.
\end{corollary}

\begin{proof}
  The assumption implies that $t_p=r+\widetilde{t}_p$, where
  $\widetilde{t}_p$ is the trace function of a Fourier sheaf with
  conductor $\leq \cond(\sheaf{F}_p)$, and we may combine
  Proposition~\ref{pr-l2}, applied to $\widetilde{t}_p$, with the usual
  mean-ergodic theorem to derive the convergence of $v_p$ to $r\pi$ in
  the topology of pointwise convergence.  Moreover, since
  $\|v_p\|\leq \cond(\sheaf{F}_p)$ for all $p$, the family $(v_p)$ is
  equicontinuous, and hence the convergence holds in fact uniformly over
  compact subsets of $L^2(X,\mu)$ (see~\cite[p.\,16,\,th.\,1]{TGX}).
\end{proof}

\begin{example}\label{ex-sp}
  Let $S_p$ be the set of quadratic residues modulo $p$. Assume that
  $f$ is $\mu$-ergodic, so that the $1$-eigenspace is spanned by the
  constant function $1$ and $\pi(\varphi)=\int_X\varphi d\mu$ for all
  $\varphi\in L^2(X,\mu)$. We then have
$$
\frac{1}{p}\sum_{\substack{0\leq n<p\\n\in S_p}}
\varphi\circ f^n\to \frac{1}{2}\int_X \varphi\, d\mu
$$
uniformly for $\varphi\in L^2(X,\mu)$ in compact subsets of
$L^2(X,\mu)$. Indeed, the characteristic function of $S_p$, for
$p\geq 3$, is
$$
\frac{1}{2}(1+\chi_p)
$$
where $\chi_p$ is the Legendre character modulo $p$, and the latter is
the trace function of a rank $1$ non-trivial Kummer sheaf.
\par
Using~\cite[\S 6.2]{fkm2}, one can extend straightforwardly this
result by replacing $S_p$ with the set $S_{q,p}=q(\Ff_p)$ of the
values modulo $p$ of a fixed polynomial $q\in\Zz[X]$ (except that the
leading constant $1/2$ might be replaced by a value depending on $p$).
\end{example}

\section{Weakly-mixing systems}
\label{sec-mean-general}

It remains to prove Theorem~\ref{th-main} under Assumption~(b). By
linearity, it suffices to treat the case of trace functions of
(geometrically irreducible) sheaves modulo primes $p\in\sfP$ with
bounded conductor.

We keep the notation of the previous section concerning the dynamical
system and the family $(\sheaf{F}_p)$ as well as the operator $u_f$.
We write
$$
\alpha_p=\frac{1}{p}\sum_{0\leq n<p}t_p(n).
$$

We use a suitable decomposition of the trace function~$t_p$. We write
$t_p=t_p^{AS}+\widetilde{t}_p$, where $t_p^{AS}$ is the Artin-Schreier
component and $\widetilde{t}_p$ is the trace function of a Fourier sheaf
$\widetilde{\sheaf{F}}_p$ with bounded conductor. Using the Riemann
Hypothesis, we can express further
$$
t_p^{AS}=\alpha_p+\widetilde{t}_p^{AS}+O(p^{-1/2}),
$$
for all $p$, where $\widetilde{t}_p^{AS}$ is the trace function of
an Artin-Schreier sheaf $\sheaf{A}_p$ with no trivial geometrically
irreducible component and with bounded conductor, and where the
implied constant depends only on $\cond(\sheaf{F}_p)$.

\begin{proposition}\label{pr-wm}
  Suppose that the system $(X,\mu,f)$ is ergodic and that the
  Kronecker factor of $(X,\mu,f)$ is trivial, or in other words that
  $(X,\mu,f)$ is weakly mixing.
  
  The endomorphisms
  $$
  v_p-\alpha_p \pi=\frac{1}{p}\sum_{0\leq n<p} t_p(n)\ u_f^n-\alpha_p
  \pi
  $$
  of $L^2(X,\mu)$ converge to $0$ in the topology of uniform
  convergence on compact subsets, and
  $$
  \frac{1}{p}\sum_{0\leq n<p} t_p(n)
  \varphi(f^n(x))
  -\alpha_p\to 0
  $$
  for almost all~$x$.
\end{proposition}

\begin{proof}
  Using the decomposition
  $$
  t_p=\alpha_p+\widetilde{t}_p+\widetilde{t}_p^{AS},
  $$
  we have
  $$
  \frac{1}{p}\sum_{0\leq n<p} \widetilde{t}_p(n)\ u_f^n(\varphi)\to 0
  $$
  by Proposition~\ref{pr-l2} applied to the sheaves
  $\widetilde{\sheaf{F}}_p$, and
  $$
  \frac{1}{p}\sum_{0\leq n<p} \alpha_p\ u_f^n(\varphi)-
  \alpha_p\pi(\varphi) \to 0
  $$
  by the classical mean-ergodic
  theorem~\cite[Th.\,2.21]{einsiedler-ward}.

  Similarly, the pointwise convergence holds almost surely for these two
  components by Theorem~\ref{th-main} and the classical pointwise
  ergodic theorem (see, e.g.,~\cite[Th.\,2.30]{einsiedler-ward}).

  We now use the assumption that the dynamical system is weakly mixing:
  a result of Bourgain (the uniform Wiener--Wintner Theorem, see the
  proof by Assani~\cite[Th.\,6]{assani}) then implies that
  $$
  \frac{1}{p}\sum_{0\leq n<p}
  e(n\theta)\varphi(f^n(x))\to 0
  $$
  for almost all~$x$, uniformly for $\theta\in [0,1]$.
  Since the trace function of $\widetilde{t}_p^{AS}$ is a finite linear,
  combination with coefficients of size~$1$, of additive characters
  $n\mapsto e(an/p)$, it follows that
  $$
  \frac{1}{p}\sum_{0\leq n<p}
  \widetilde{t}_p^{AS}(n)\varphi(f^n(x))\to 0
  $$
  almost surely (although that the number of such additive characters
  may depend on~$p$, this doesn't affect this argument).
  
  This concludes the proof of the pointwise part of
  Theorem~\ref{th-main} for weakly mixing systems. The mean-ergodic
  convergence follows by the dominated convergence theorem.
\end{proof}

Besides this proof, we now give an alternative argument for the
mean-ergodic theorem in this case, which does not use the uniform
Wiener--Wintner Theorem. This can be skipped (we include it since these
are informal notes, and the arguments were elaborated before we were
aware of this result of Bourgain).

We will need the following definition to state the basic technical fact.

\begin{definition}\label{def-emphatic}
  Let $\theta\in \Rr/\Zz$ and let $(a_p)_p$ be a sequence of integers,
  indexed by an infinite set of primes. We say that $a_p/p$
  \emph{converges emphatically} to $\theta$ if
  $$
  \limsup_{p\to+\infty} p \Bigl|\frac{a_p}{p}-\theta\Bigr|<+\infty,
  $$
  and if moreover no subsequence of $(p|\tfrac{a_p}{p}-\theta|)_p$
  converges to a positive integer.
\end{definition}

\begin{remark}
  If $a_p/p$ converges emphatically to $\theta$, then $a_p/p$
  converges to $\theta$. If $\theta=0$, then one sees that the
  condition means that $a_p=0$ for all but finitely many $p$.
  \par
  For any $\theta_0\in\Rr/\Zz$, there is a sequence $(a_p)$ indexed by
  primes such that $(a_p)$ converges emphatically to $\theta_0$, by
  taking $a_p/p$ the closest to $\theta_0$, so that
  $p|\tfrac{a_p}{p}-\theta_0|<1$.
\end{remark}

\begin{proposition}
  \label{pr-as}
  Let $(a_p)$ be a sequence of integers indexed by an infinite subset
  of primes. Assume that $a_p/p$ converges to $\theta_0$ in $\Rr/\Zz$.
  Let $\varphi\in L^2(X,\mu)$ of norm $1$ and define
  $$
  \psi_p=\frac{1}{p}\sum_{0\leq n<p} e\Bigl(-\frac{na_p}{p}\Bigr) \
  u_f^n(\varphi).
  $$
  \par
  \emph{(1)} Suppose that $\theta_0\not=0$ in $\Rr/\Zz$. If the
  sequence $(\|\psi_p\|)$ converges to a non-zero number, then the
  sequence $(a_p/p)$ converges emphatically to $\theta_0$, and
  $e(\theta_0)$ is an eigenvalue of $u_f$.
  \par
  \emph{(2)} Suppose that $\theta_0=0$ and $p\nmid a_p$ for all
  $p$. Then $\psi_p\to 0$.
\end{proposition}

\begin{proof}
  As before, let $\nu$ be the spectral measure of the unitary operator
  $u_f$ relative to the unit vector $\varphi$.
  We obtain
  $$
  \|\psi_p\|^2= \int_{\Rr/\Zz}\Bigl| \frac{1}{p}\sum_{0\leq
    n<p}e\Bigl(n\Bigl(\theta-\frac{a_p}{p}\Bigr)\Bigr)
  \Bigr|^2d\nu(\theta)=
  \int_{\Rr/\Zz}\frac{1}{p}F_p\Bigl(\theta-\frac{a_p}{p}\Bigr)d\nu(\theta),
  $$
  where $F_p$ is the Fejér kernel: $F_p(0)=p$ and
  $$
  F_p(\theta)=
  \frac{1}{p}\Bigl(\frac{\sin(\pi p \theta)}{\sin(\pi\theta)}\Bigr)^2
  $$
  for $\theta\not=0$.
  \par
  Recall that $0\leq F_p\leq p$, so $p^{-1}|F_p|\leq 1$. Moreover,
  $F_p(\theta)\to 0$ uniformly on the complement of any neighborhood of
  $0$ in $\Rr/\Zz$. Thus, using the limit assumption
  $a_p/p\to \theta_0$, we have
  $$
  F_p\Bigl(\theta-\frac{a_p}{p}\Bigr)\to 0
  $$
  as $p\to+\infty$ for any fixed $\theta\not=\theta_0$, and a fortiori
  we have the same limit after dividing the left-hand side by $p$.
  \par
  We first prove (2), and thus assume that $\theta_0=0$ and
  $p\nmid a_p$. Then
  $$
  F_p(\theta_0-\tfrac{a_p}{p})=F_p(-\tfrac{a_p}{p})=0,
  $$
  for all~$p$, hence we obtain $\|\psi_p\|\to 0$ by the dominated
  convergence theorem.
  \par
  Now we prove (1), and assume that $\theta_0\not=0$ and that
  $\|\psi_p\|$ converges to a non-zero number. If the sequence
  $(p|\tfrac{a_p}{p}-\theta_0|)$ is unbounded, then using the assumption
  $\theta_0\not=0$ and the formula defining $F_p$, we see that there is
  a subsequence of primes such that
  $$
  \frac{1}{p}F_p\Bigl(\theta_0-\frac{a_p}{p}\Bigr) \ll
  \frac{1}{p^2\bigl|\theta_0-\frac{a_p}{p}\bigr|^2}\to 0
  $$
  as $p\to +\infty$. We conclude using the dominated convergence theorem
  that $\|\psi_p\|^2\to 0$ along this subsequence, contrary to the
  assumption.
  \par
  Thus we have
  $$
  \sup_{p\to +\infty}p\Bigl|\frac{a_p}{p}-\theta_0\Bigr|=C<+\infty.
  $$
  Consider any subsequence of primes where the sequence
  $(p|\tfrac{a_p}{p}-\theta_0|)_p$ converges to some real number
  $c\geq 0$. Then
  $$
  \frac{1}{p}F_p\Bigl(\theta_0-\frac{a_p}{p}\Bigr)\to
  \Bigl(\frac{\sin(\pi c)}{\pi c}\Bigr)^2,
  $$
  hence, along this subsequence, the dominated convergence theorem
  gives
  $$
  \lim_{p\to+\infty}\|\psi_p\|^2= \Bigl(\frac{\sin(\pi c)}{\pi
    c}\Bigr)^2 \nu(\{\theta_0\}).
  $$
  
  Since we assumed that the left-hand side exists and is non-zero, we
  conclude that $\theta_0$ is an atom of $\nu$. As is well-known, this
  implies that $e(\theta_0)$ is an eigenvalue of $u_f$ (because it
  implies that the spectral projector relative to~$\{e(\theta_0)\}$ is
  non-zero; see, e.g.,~\cite[p.\,279,\,Cor.]{TS345}).
\end{proof}

\begin{corollary}\label{cor-wm}
  Suppose that the system $(X,\mu,f)$ is ergodic and that the
  Kronecker factor of $(X,\mu,f)$ is trivial, or in other words that
  $(X,\mu,f)$ is weakly mixing. Let
  $$
  \alpha_p=\frac{1}{p}\sum_{0\leq n<p}t_p(n).
  $$
  Then the endomorphisms
  $$
  v_p-\alpha_p \pi=\frac{1}{p}\sum_{0\leq n<p} t_p(n)\ u_f^n-\alpha_p
  \pi
  $$
  of $L^2(X,\mu)$ converge to $0$ in the topology of uniform
  convergence on compact subsets.
\end{corollary}

\begin{proof}
  Since $|\alpha_p|\leq \cond(\sheaf{F}_p)$, the family of
  endomorphisms $v_p-\alpha_p\pi$ is equicontinuous, hence it suffices
  to prove pointwise convergence to $0$ for all
  $\varphi\in L^2(X,\mu)$. We may further assume that $\varphi$ has
  norm $1$.
  
  We write $t_p=t_p^{AS}+\widetilde{t}_p$, where $t_p^{AS}$ is the
  Artin-Schreier component and $\widetilde{t}_p$ is the trace function
  of a Fourier sheaf $\widetilde{\sheaf{F}}_p$ with bounded
  conductor. Using the Riemann Hypothesis, we can express further
  $$
  t_p^{AS}=\alpha_p+\widetilde{t}_p^{AS}+O(p^{-1/2}),
  $$
  for all $p$, where $\widetilde{t}_p^{AS}$ is the trace function of
  an Artin-Schreier sheaf $\sheaf{A}_p$ with no trivial geometrically
  irreducible component and with bounded conductor, and where the
  implied constant depends only on $\cond(\sheaf{F}_p)$. Then we have
  $$
  \frac{1}{p}\sum_{0\leq n<p} \widetilde{t}_p(n)\ u_f^n(\varphi)\to 0
  $$
  by Proposition~\ref{pr-l2} applied to the sheaves
  $\widetilde{\sheaf{F}}_p$, and
  $$
  \frac{1}{p}\sum_{0\leq n<p} \alpha_p\ u_f^n(\varphi)-
  \alpha_p\pi(\varphi)
  \to 0
  $$
  by the classical mean-ergodic
  theorem~\cite[Th.\,2.21]{einsiedler-ward}.
  \par
  We are now done unless $\sheaf{A}_p$ has rank $\geq 1$ for an
  infinite sequence of primes. We now assume this and consider only
  such primes. Let
  $$
  \psi_p= \frac{1}{p}\sum_{0\leq n<p} \widetilde{t}^{AS}_p(n)\
  u_f^n(\varphi).
  $$
  
  The sequence $(\|\psi_p\|)_p$ is bounded by the maximum of the ranks
  of the sheaves~$\sheaf{A}_p$. Let
  $c\geq 0$ be a limiting value, obtained for a subsequence of primes
  which we omit from the notation. Assume that $c>0$. By passing to a
  further subsequence, we may assume that the rank of $\sheaf{A}_p$ is
  a constant $r\geq 1$. We have geometric isomorphisms
  $$
  \sheaf{A}_p\simeq \bigoplus_{j=1}^r\sheaf{L}_{\psi(-a_{p,j}x)}
  $$
  for some integers $0<a_{p,j}<p$. There must exist some fixed $j$
  such that the norm of
  $$
  \widetilde{\psi}_p=\frac{1}{p}\sum_{0\leq n<p}
  e\Bigl(-\frac{a_{p,j}n}{p}\Bigr) \ u_f^n(\varphi)
  $$
  does not converge to $0$, as otherwise we would obtain $c=0$.  We
  may then assume, again by passing to a subsequence, that
  $-a_{p,j}/p$ converges to some $\theta_0\in\Rr/\Zz$. Since
  $\widetilde{\psi}_p$ does not converge to $0$, we have
  $\theta_0\not=0$ by Proposition~\ref{pr-as}, (2). 
  
  Now, by definition (see~\cite[Th. 2.36 or \S 6.4]{einsiedler-ward}),
  the assumption on $(X,\mu,f)$ means that $u_f$ has no eigenvalue
  different from $1$ (and that $1$ is an eigenvalue of multiplicity
  one). We have then a contradiction to Proposition~\ref{pr-as}, (1).
  This means that all limit points of the bounded sequence
  $(\|\psi_p\|)_p$ are equal to $0$, hence it converges to $0$.
\end{proof}

Using linearity, this corollary implies Theorem~\ref{th-main}, (1) under
Assumption (b). 

\begin{example}
  Examples of weakly mixing systems $(X,\mu,f)$ are Bernoulli shifts,
  ergodic automorphisms of compact abelian groups (e.g., elements of
  $\SL_d(\Zz)$ acting on $(\Rr/\Zz)^d$ which have no root of unity as
  an eigenvalue) or the Gauss map in the theory of continued
  fractions~\cite{philipp}.

  Another important class arises in homogeneous dynamics. Let $G$ be a
  locally compact group, $\Gamma$ a lattice in $G$ and consider the
  action of $G$ on $X=\Gamma\backslash G$. Denote by $\mu_X$ the
  $G$-invariant probability measure on $X$. Assume that the action is
  mixing~\cite[\S\,8.1]{einsiedler-ward}. Let $x\in G$ be such that
  $x^n\to +\infty$ in $G$ as $n\to +\infty$. Then defining
  $f(\Gamma y)=\Gamma yx$, we obtain a system $(X,\mu_X,f)$ that is
  mixing by definition, hence weakly mixing. This applies for instance
  to $G=\SL_2(\Rr)$ and $x$ a non-trivial unipotent element.
\end{example}




\section{The topological case}\label{sec-sarnak}

In this section, we prove Theorem~\ref{th-sarnak}. Thus let $X$ be a
compact topological space and $f\colon X\to X$ a continuous map, such
that the topological entropy $h(f)$ is zero (see, e.g,~\cite[\S
4]{einsiedler-schmidt} for an introduction to topological
entropy). Let $\varphi\colon X\to \Cc$ be continuous and $x\in X$.
The goal is to find conditions on a family $(\sheaf{F}_p)$ of sheaves
modulo $p$ with bounded conductor which imply that
$$
\lim_{p\to+\infty} \frac{1}{p}\sum_{0\leq n<p}
t_p(n)\varphi(f^n(x))=0,
$$
with no exceptions or sparseness assumption. The claim of
Theorem~\ref{th-sarnak} is that this is the case when the family
consists of Kloosterman sheaves or Kummer sheaves associated to real
characters, for which $t_p(n)=\Kl_2(n;p)$ or~$t_p(n)=(\tfrac{n}{p})$,
respectively.

The proof is in fact a straightforward adaptation of the combinatorial
argument that shows that decay of multiple correlations of the Möbius
function (what is called the Chowla conjecture) implies Sarnak's
conjecture, as presented e.g. on Tao's blog~\cite{tao}, and extends to
a certain class of sheaves introduced in~\cite{sop} under the name of
``bountiful sheaves'' (\cite[Def.\,1.2]{sop}). For a clearer
perspective, we make the following definition:

\begin{definition}\label{def-monodromy-entropy}
  Let $(\sheaf{F}_p)$ be a family of sheaves modulo $p$ with bounded
  conductor. We say that it has \emph{positive monodromy-entropy} if for
  any integers $k\geq 1$ and $H\geq 1$, the number $N_p(k,H)$ of
  tuples of non-negative integers
  $(h_1,\ldots, h_k,h'_1,\ldots, h'_k)$ with $h_i$, $h_j\leq H$ such
  that
  $$
  \bigotimes_{i=1}^k [+h_i]^* \sheaf{F}_p\otimes \bigotimes_{i=1}^k
  [+h'_i]^* \dual(\sheaf{F}_p)
  $$
  contains a geometrically trivial irreducible component satisfies
  $$
  N_p(k,H)\ll (2k)^{k}H^{k}.
  $$
\end{definition}

A key point in this definition is that the number $N_p(k,H)$ is
bounded independently of $p$, but it is also important that the
exponent of~$H$ is no larger than~$k$.

Here is our general statement:

\begin{proposition}\label{pr-sarnak-2}
  Let $(\sheaf{F}_p)_p$ be a family of sheaves modulo $p$ with
  positive monodromy-entropy and bounded conductor.

  Let $X$ be a locally compact topological space and $f\colon X\to X$
  a continuous map. Assume that either $X$ is compact or that~$X$ is a
  metric space and $f$ uniformly continuous.

  Assume that the topological entropy of $f$ is zero.  Then for all
  bounded\footnote{Check} continuous functions
  $\varphi\colon X\to \Cc$ and all $x\in X$, we have
  \begin{equation}\label{eq-lim-sarnak}
    \lim_{p\to+\infty} \frac{1}{p}\sum_{0\leq n<p}
    t_p(n)\varphi(f^n(x))=0.
  \end{equation}
\end{proposition}

This implies Theorem~\ref{th-sarnak} in view of the following lemma:

\begin{lemma}\label{lm-zero}
  \emph{(1)} If $(\sheaf{F}_p)_p$ is a family of bountiful sheaves,
  then it has positive monodromy-entropy.
  \par
  \emph{(2)} If $(\sheaf{F}_p)_p$ is a family such that $\sheaf{F}_p$ is
  a non-trivial Kummer sheaf for all $p$, then it has positive
  monodromy-entropy. 
\end{lemma}

\begin{proof}
  In case (1), this follows immediately
  from~\cite[Def.\,1.2,Th.\,1.5]{sop} and elementary combinatorics,
  taking into account the definitions of normal and $r$-normal tuples
  (see~\cite[Def.\,1.3]{sop}).
  \par
  In case (2), if $\sheaf{F}_p=\sheaf{L}_{\chi}$, where $\chi$ has order
  $d\mid p-1$, with $d\geq 2$, then note that
  $$
  \bigotimes_{i=1}^k [+h_i]^* \sheaf{F}_p\otimes
  \bigotimes_{i=1}^k [+h'_i]^* \dual(\sheaf{F}_p)
  =\sheaf{L}_{\chi(G/H)},
  $$
  where $G$ and $H$ are the polynomials
  $$
  G=\prod_{i=1}^k(X+h_i),\quad\quad H=\prod_{j=1}^k(X+h'_j).
  $$
  This contains a geometrically trivial component if and only if $G/H$
  is a $d$-th power of a rational function. The bound on $N_p(k,H)$ is
  therefore clear (the worse case is when $d=2$, and then the estimate
  is the same as that for normal tuples, as
  in~\cite[Def.\,1.5,\,(1)]{sop}).
\end{proof}

We now prove Proposition~\ref{pr-sarnak-2}, following
closely~\cite{tao}. The next statement, which provides the analogue of
decay of multiple correlations of the Möbius function, could also be
derived from the work of Perret-Gentil~\cite{p-g} in most cases of
interest.

\begin{proposition}\label{pr-zero}
  Let $(\sheaf{F}_p)_p$ be a family of sheaves modulo $p$ with
  positive monodromy-entropy and bounded conductor. Let
  $(\alpha_n)_{n\geq 0}$ be a sequence of complex numbers bounded
  by~$1$.
  \par
  Fix an integer $m\geq 1$. There exists a absolute constant $C>0$
  such that, for any $\eps>0$, we have
  $$
  \frac{1}{p}\Bigl|\{0\leq n<p\,\mid\, \Bigl|\sum_{0\leq
    i<m}t_p(n+i)\alpha_i\Bigr|\geq \eps m\}\Bigr| \leq
  C\exp\Bigl(-\frac{\eps^2 m}{C}\Bigr)+O(\eps^{-\eps^2m}p^{-1/2}),
  $$
  where the implied constant depends only on the conductor of
  $(\sheaf{F}_p)$.
\end{proposition}

\begin{proof}
  Let $k\geq 1$ be an integer to be chosen later. We have
  $$
  \frac{1}{p}|\{0\leq n<p\,\mid\, \Bigl|\sum_{0\leq
    i<m}t_p(n+i)\alpha_i\Bigr|\geq \eps m\}| \leq \frac{1}{(\eps m)^{2k}}
  \frac{1}{p}\sum_{0\leq n<p} \Bigl|\sum_{0\leq
    i<m}t_p(n+i)\alpha_i\Bigr|^{2k}.
  $$
  
  Since $|\alpha_i|\leq 1$, if we expand the right-hand side, we obtain
  the upper bound
  $$
  \frac{1}{(\eps m)^{2k}}
  \sumsum_{\substack{0\leq a_1,\ldots,a_k<m\\0\leq b_1,\ldots,b_k<m}}
  \frac{1}{p}\sum_{0\leq n<p}
  t_p(n+a_1)\cdots t_p(n+a_k)
  \overline{t_p(n+b_1)\cdots t_p(n+b_k)}.
  $$

  Because of the monodromy-entropy assumption and the Riemann Hypothesis
  (see~\cite[Prop. 1.1]{sop}), the inner sum is $\ll p^{-1/2}$, with
  implied constant depending only on $k$ and $\cond(\sheaf{F}_p)$, for
  all but $\leq (2k)^{k}m^{k}$ tuples $(a_i,b_j)$.

  It follows that
  $$
  \frac{1}{p}|\{0\leq n<p\,\mid\, \Bigl|\sum_{0\leq
    i<m}t_p(n+i)\alpha_i\Bigr|\geq \eps m\}| \leq \frac{1}{(\eps
    m)^{2k}} \Bigl((2k)^{k}m^{k}+O\Bigl(
  \frac{m^{2k}}{\sqrt{p}}\Bigr)\Bigr)
  $$
  where the implied constant depends only on $k$ and
  $\cond(\sheaf{F}_p)$. Taking $k$ to be the closest integer
  $\leq \eps^2 m/10$, the result follows.
\end{proof}

The crucial feature of this estimate is the fact that the first term
decays exponentially with respect to~$m$.
We sketch the argument for completeness. We may assume that $\varphi$
is real-valued and bounded by $1$. Let $0<\eps<1$ be fixed, and let
$(\varphi_{\eps}(n))$ be a sequence with values in $\Zz\eps$, such
that $|\varphi_{\eps}(n)|\leq 1$ and
$|\varphi(f^n(x))-\varphi_{\eps}(n)|\leq \eps$ for all~$n\geq 0$. Fix
an integer $m\geq 1$. Define $\kappa_{\eps}(m)$ so that the tuples
\begin{equation}\label{eq-tuple}
  (\varphi_{\eps}(n),\ldots,\varphi_{\eps}(n+m-1))\in (\Zz\eps\cap
  [-1,1])^m
\end{equation}
take $\exp(\kappa_{\eps}(m))$ values as $n$ ranges over the
non-negative integers. The fact that the topological entropy of~$f$ is
zero (i.e, the sequence $(\varphi(f^n(x)))_{n}$ is deterministic)
implies that
$$
\lim_{m\to +\infty}\frac{\kappa_{\eps}(m)}{m}=0.
$$

Let $p$ be a large prime. For any tuple~(\ref{eq-tuple}), say
$(\alpha_0,\ldots, \alpha_{m-1})$, Proposition~\ref{pr-zero} shows
that we have
$$
\frac{1}{p}\Bigl|\{0\leq n<p\,\mid\, \Bigl|\sum_{0\leq
  i<m}t_p(n+i)\alpha_i\Bigr|\geq \eps m\}\Bigr| \leq
C\exp\Bigl(-\frac{\eps^2 m}{C}\Bigr)+O(\eps^{-\eps^2m}p^{-1/2}),
$$
and hence
\begin{multline*}
  \frac{1}{p}\Bigl|\{0\leq n<p\,\mid\, \Bigl|\sum_{0\leq
    i<m}t_p(n+i)\varphi_{\eps}(n+i)\Bigr|\geq \eps m\}\Bigr| \\
  \leq C\exp\Bigl(-\frac{\eps^2
    m}{C}+\kappa_{\eps}(m)\Bigr)+O(\eps^{-\eps^2m}e^{\kappa_{\eps}(m)}p^{-1/2}).
\end{multline*}
Since $\kappa_{\eps}(m)/m\to 0$, we may take $m$ large enough
(depending on $\eps$) so that this implies
$$
 \frac{1}{p}\Bigl|\{0\leq n<p\,\mid\, \Bigl|\sum_{0\leq
    i<m}t_p(n+i)\varphi_{\eps}(n+i)\Bigr|\geq \eps m\}\Bigr| \\
  \leq \eps+o(1)
$$
as $p\to +\infty$. But then we deduce that
$$
\Bigl| \frac{1}{p}\sum_{0\leq n<p}\frac{1}{m}\sum_{0\leq
  i<m}t_p(n+i)\varphi_{\eps}(n+i)\Bigr|\leq 2\eps+o(1)
$$
because $|\varphi_{\eps}|\leq 1$ (write the average as the sum of a
term where it is $>\eps$, handled by the above inequality, and one
where it is $\leq \eps$, which has a contribution $\leq \eps$).
\par
Now notice that for $0\leq i<m$, we have
$$
\frac{1}{p}\sum_{0\leq n<p}t_p(n+i)\varphi_{\eps}(n+i)
=\frac{1}{p}\sum_{0\leq
  n<p}t_p(n)\varphi_{\eps}(n)+O\Bigl(\frac{m\cond(\sheaf{F}_p)}{p}\Bigr)
$$
with an absolute implied constant, so we get
$$
\Bigl| \frac{1}{p}\sum_{0\leq n<p}t_p(n)\varphi_{\eps}(n)\Bigr|\leq
2\eps+o(1),
$$
hence
$$
\Bigl| \frac{1}{p}\sum_{0\leq n<p}t_p(n)\varphi(f^n(x))\Bigr|\leq
3\eps+o(1).
$$
The limit~(\ref{eq-lim-sarnak}) follows.

\section{Mean-ergodic theorems in $L^r$}\label{sec-lr}

The goal of this section is to see if one can extend the mean-ergodic
theorem to the spaces $L^r(X,\mu)$ when $r>2$. We will achieve this
goal, however, only for sheaves satisfying the same extra condition
which appeared in the previous section.

\begin{proposition}
  Suppose that $(\sheaf{F}_p)$ is a family of sheaves with positive
  monodromy-entropy.  Let $r>2$ be fixed. The endomorphisms
  $$
  \widetilde{v}_p=\frac{1}{p}\sum_{0\leq n<p} t_p(n)\ u_f^n
  $$
  of $L^{r}(X,\mu)$ converge to $0$ as $p\to +\infty$.
\end{proposition}

\begin{proof}
  Using monotonicity, as in Corollary~\ref{cor-lr12}, it is enough to
  prove this when $r=2k$ for some integer $k\geq 2$ to deduce it for
  $r\leq 2k$.
  \par
  Let $\varphi\in L^{2k}(X,\mu)$ and denote
  $\psi_p=\widetilde{v}_p(\varphi)$. We have
\begin{multline*}
  \|\psi_p\|_{2k}^{2k}= \frac{1}{p^{2k}} \sum_{\substack{n_1,\ldots,
      n_k\\0\leq n_i<p}} \sum_{\substack{m_1,\ldots,m_k\\0\leq m_j<p}}
  t_p(n_1)\cdots t_p(n_k)
  \overline{t_p(m_1)\cdots t_p(m_k)}\\
  \langle u_f^{n_1}(\varphi)\cdots u_f^{n_k}(\varphi),
  u_f^{m_1}(\varphi)\cdots u_f^{m_k}(\varphi) \rangle.
\end{multline*}
Since $u_f$ is isometric, we have
$$
\langle u_f^{n_1}(\varphi)\cdots u_f^{n_k}(\varphi),
u_f^{m_1}(\varphi)\cdots u_f^{m_k}(\varphi) \rangle=
\langle
\varphi\cdots u_f^{n_k-n_1}(\varphi), u_f^{m_1-n_1}(\varphi)\cdots
u_f^{m_k-n_1}(\varphi) \rangle.
$$
Hence, we may sum over $h=n_1$ first, obtaining
\begin{multline*}
  \|\psi_p\|_{2k}^{2k}= \frac{1}{p^{2k}} \sum_{n_2,\ldots,
    n_k}\sum_{m_1,\ldots,m_k} \langle \varphi\
  u_f^{n_2}(\varphi)\cdots u_f^{n_k}(\varphi),
  u_f^{m_1}(\varphi)\cdots u_f^{m_k}(\varphi) \rangle
  \\
  \sum_h t_p(h)t_p(h+n_1)\cdots t_p(h+n_k) \overline{t_p(h+m_1)\cdots
    t_p(h+m_k)},
\end{multline*}
where the sum is over integers $0\leq h<p$ such that
$$
0\leq h+n_i<p,\quad\quad 0\leq h+m_j<p
$$
for $2\leq i\leq k$ and $1\leq j\leq k$, respectively. This is a sum
over an interval of length $<p$.  The assumption on $\sheaf{F}_p$ then
implies that
$$
\sum_h t_p(h)t_p(h+n_1)\cdots t_p(h+n_k) \overline{t_p(h+m_1)\cdots
  t_p(h+m_k)}\ll p^{1/2}(\log p),
$$
where the implied constant depends only on $\cond(\sheaf{F}_p)$ and
$k$, unless $(0,n_2,\ldots,n_k)$ is a permutation of
$(m_1,\ldots,m_k)$ (see~\cite[Th. 1.5]{sop}). This occurs for
$\ll p^{k-1}$ tuples $(n_2,\ldots,m_k)$, and for these we have a bound
$\ll p$ for the sum, where the implied constant depends only on
$\cond(\sheaf{F}_p)$ and $k$. Thus we derive
$$
\|\psi_p\|_{2k}^{2k}\ll p^{-1/2}(\log p)+p^{-k-1}.
$$
\par
This shows that $\|\widetilde{v}_p\|\to 0$ in the space of
endomorphisms of $L^{2k}(X,\mu)$, and concludes the proof.
\end{proof}

\begin{remark}
  Analyzing the proof of the proposition further, we can reach a
  stronger conclusion and indeed derive a slightly stronger pointwise
  statement than the one in Theorem~\ref{th-main}, although under
  assumptions that are reasonable in principle, but difficult to
  check.

  We take the case $k=2$ of the proposition, and rewrite the first
  steps above: for $\varphi\in L^4(X,\mu)$, we have
  $$
  \|\psi_p\|_{4}^{4}= \frac{1}{p^{4}} \sum_{b} \sum_{c,d} \langle
  \varphi\ u_f^{b}(\varphi),u_f^{c}(\varphi) u_f^{d}(\varphi) \rangle
  \sum_a t_p(a)t_p(a+b)\overline{t_p(a+c)t_p(a+d)}.
  $$
  We rewrite the sum in the form
  $$
  \|\psi_p\|_{4}^{4}= \frac{1}{p^{7/2}} \sum_{c,d} \langle \varphi\
  u_f^{b}(\varphi),u_f^{c}(\varphi) u_f^{d}(\varphi) \rangle
  \tau_{c,d}(b)
  $$
  where
  $$
  \tau_{c,d}(b)= \frac{1}{\sqrt{p}} \sum_a
  t_p(a)t_p(a+b)\overline{t_p(a+c)t_p(a+d)},
  $$
  hence
  $$
  \|\psi_p\|_{4}^{4}= \frac{1}{p^{5/2}} \sum_{c,d} \langle
  w_{p,c,d}(\varphi), \bar{\varphi}\ u_f^{c}(\varphi) u_f^{d}(\varphi)
  \rangle
  $$
  where
  $$
  w_{p,c,d}(\varphi)=\frac{1}{p}\sum_{0\leq b<p}\tau_{c,d}(b)\, \varphi
  \circ f^b.
  $$
  
  Now assume that $\varphi\in L^6(X,\mu)$, which implies that
  $\bar{\varphi}\ u_f^{c}(\varphi) u_f^{d}(\varphi)$ belongs to
  $L^2(X,\mu)$ and has norm $\ll 1$. Assume moreover that the family
  $(\sheaf{F}_p)$ satisfies the condition that \emph{for most $(c,d)$,
    with $\ll p$ exceptions, the function $\tau_{c,d}$ is a trace
    function of a Fourier sheaf, with weights $\leq 0$}. Note then that,
  by Sawin's Quantitative Sheaf Theory
  (see~\cite[Th.\,1.1,\,Cor.\,7.4]{qst}), the conductor of~$\tau_{c,d}$
  is $\ll 1$ for all~$p$.

  Under these conditions, by Proposition~\ref{pr-l2}, we obtain
  $$
  \|w_{p,c,d}(\varphi)\|_2^2\ll p^{-1/2}\log p,
  $$
  for most $(c,d)$, and hence conclude that
  $$
  \|\psi_p\|_4^4\ll p^{-1}\log p.
  $$
  
  If we assume that the family $(\sheaf{F}_p)$ is indexed by a set of
  primes $\sfP$ such that
  $$
  \sum_{p\in\sfP} \frac{\log p}{p}<+\infty,
  $$
  then this result means that
  $$
  \sum_p \|\psi_p\|_4^4<+\infty,
  $$
  or in other words that the non-negative function
  $$
  \sum_p |\psi_p|^4
  $$
  is integrable on $X$. This imples that
  $\psi_p=\widetilde{v}_p(\varphi)$ converges almost everywhere to
  $0$, a pointwise theorem. This is a bit stronger than the pointwise
  part of Theorem~\ref{th-main}, but the latter does not require any
  extra condition, and hence we do not pursue the verification that
  the assumption above holds in reasonable situations. 
\end{remark}

\section{Maximal inequalities in $L^2$}

We now consider maximal inequalities in $L^2$, i.e., we endeavor to
estimate functions like
$$
M\varphi\colon x\mapsto \sup_{p}\Bigl| \frac{1}{p}\sum_{0\leq n<p}
t_p(n)\ \varphi(f^n(x))\Bigr|
$$
in $L^2$-norm, where we have fixed the dynamical system $(X,\mu,f)$ and
the family of sheaves $(\sheaf{F}_p)$ with bounded conductor, and with
trace functions $t_p$. In fact, we will need to restrict the supremum to
sparse subsets of the primes, and so we use the notation
$$
M_{\sfP}(\varphi)(x)=\sup_{p\in\sfP}\Bigl|
\frac{1}{p}\sum_{0\leq n<p} t_p(n)\ \varphi(f^n(x))\Bigr|
$$
for any set $\sfP$ of primes. We write
$$
s(\sfP)=\sum_{p\in\sfP}\frac{(\log p)^2}{p},
$$
which is finite if and only if~$\sfP$ is sparse.


\begin{proposition}\label{pr-maximal}
  Suppose that $(\sheaf{F}_p)_{p}$ is an almost Fourier family
  \textup{(Definition~\ref{def-af})} with mean $r\geq 0$.  Suppose
  further that $\sfP$ is sparse. Let $\varphi\in L^2(X,\mu)$. We have
  $$
  \|M_{\sfP}\varphi\|_2\leq C_2\|\varphi\|_2
  $$
  for some constant $C_2$ depending only on the conductor of
  $(\sheaf{F}_p)$ and on $s(\sfP)$.
\end{proposition}

The method that we use is a direct adaptation of that of
Bourgain~\cite[\S 2, \S 3]{bourgain3} (it is in fact much simpler).  In
the remainder of this section, we fix the sparse set $\sfP$, and we will
omit it from the notation unless it is required for context.
\par
The first step is to transfer the problem to $\Zz$.  For any bounded
function $\varpi$ on $\Zz$, we define
$\widetilde{M}(\varpi)\colon \Zz\to\Cc$ by
$$
\widetilde{M}(\varpi)(k)= \sup_{p\in\sfP}\
\Bigl|\frac{1}{p}\sum_{0\leq n<p} t_p(n)\ \varpi(k+n)\Bigr|.
$$

\begin{lemma}\label{lm-transfer}
  Suppose that there exists $C_3\geq 0$, depending only on the
  conductor of $(\sheaf{F}_p)$ and on $s(\mathsf{F})$, such that
  $$
  \|\widetilde{M}\varpi\|_2\leq C_3\|\varpi\|_2
  $$
  for all $\varpi$ bounded on $\Zz$. Then Proposition~\ref{pr-maximal}
  holds with $C_2=C_3$.
\end{lemma}

\begin{proof}
  We use the classical method of transfer to $\Zz$.  It suffices to
  prove that for all $P\geq 2$ and all $\varphi\in L^{\infty}(X,\mu)$,
  we have
  $$
  \|M^P\varphi\|_2\leq 2C_3\|\varphi\|_2
  $$
  where 
  $$
  M^P(\varphi)=\sup_{p\leq P}\Bigl| \frac{1}{p}\sum_{0\leq n<p}
  t_p(n)\ (\varphi\circ f^n)\Bigr|\in L^{2}(X,\mu).
  $$
  
  Fix such a $P$ and $\varphi$ bounded and measurable on $X$.  Let
  $\lambda>1$ be a parameter and $Q=\lambda P$. Let $x\in X$. Define
  $\widetilde{\varphi}\colon\Zz\to \Cc$ by
$$
\widetilde{\varphi}(n)=
\begin{cases}
  \varphi(f^n(x))&\text{ if } 0\leq n<Q\\
  0&\text{ otherwise.}
\end{cases}
$$
Note that for any prime $p\leq P$ and $n$, $k$ such that
$0\leq n+k<Q$, we have
$$
\widetilde{\varphi}(n+k)=\varphi(f^{n+k}(x))=\varphi(f^n(f^k(x)))
$$
so that for $0\leq k<Q-P$, we get
\begin{align}
  M^P(\varphi)(f^k(x))
  &=
    \sup_{p\leq P}\Bigl| \frac{1}{p}\sum_{0\leq n<p}
    t_p(n)\ (\varphi(f^{n}(f^k(x))))\Bigr|\notag\\
  &=
    \sup_{p\leq P}\
    \Bigl|\frac{1}{p}\sum_{0\leq n<p} t_p(n)\
    \widetilde{\varphi}(k+n)\Bigr|=
    \widetilde{M}^{P}(\widetilde{\varphi})(k),
    \label{eq-transfer}
\end{align}
say. 
By assumption, we have
$\|\widetilde{M}^P(\widetilde{\varphi})\|_2\leq
\|\widetilde{M}(\widetilde{\varphi})\|_2\leq
C_3\|\widetilde{\varphi}\|_2$. This means that
$$
\sum_{k\in\Zz}|\widetilde{M}^P(\widetilde{\varphi})(k)|^2 \leq
C_3^2\sum_{n\in\Zz} |\widetilde{\varphi}(n)|^2= C_3^2\sum_{0\leq n<Q}
|\varphi(f^n(x))|^2,
$$
hence by~(\ref{eq-transfer}), we obtain
$$
\sum_{0\leq k<Q-P}|M^P(\varphi)(f^k(x))|^2 \leq C_3^2\sum_{0\leq n<Q}
|\varphi(f^n(x))|^2.
$$
This inequality is valid for all $x\in X$. After integrating over $X$,
we get
$$
\sum_{0\leq k<Q-P}\|M^P(\varphi)\circ f^k\|_2^2\leq C_3^2 \sum_{0\leq
  n<Q}\|\varphi\circ f^n\|_2^2.
$$

But $\mu$ is $f$-invariant, and therefore both sums are sums of equal
terms, which means that
$$
(\lambda -1)P\|M^P(\varphi)\|^2\leq C_3^2\lambda P\|\varphi\|^2.
$$
The result follows by taking $\lambda\to+\infty$.
\end{proof}

\begin{proof}[Proof of Proposition~\ref{pr-maximal}]
  We will prove Lemma~\ref{lm-transfer}. Since $(\sheaf{F}_p)_p$ is an
  almost Fourier family of mean $r$, we have
  $$
  t_p(n)=r+\tau_p(n)
  $$
  where $\tau_p$ is the trace function of Fourier sheaves with bounded
  conductor.
  \par
  Let $\varpi$ be a function on $\Zz$ with finite support.  We denote
  by
  $$
  \widehat{\varpi}(\theta)=\sum_{k\in\Zz}\varpi(k)e(-k\theta),
  $$
  and
  $$
  \widehat{v}_p(\theta)=\frac{1}{p}\sum_{0\leq n<p}
  \tau_p(n)e(n\theta)
  $$
  the Fourier transforms on $\Rr/\Zz$ of the function $\varpi$ and of
  the discrete measures corresponding to the average 
  $\tau_p(n)$. We have
  \begin{equation}\label{eq-convol}
    \frac{1}{p}\sum_{0\leq n<p}\tau_p(n) \varpi(n+k)= \int_{\Rr/\Zz}
    \widehat{\varpi}(\theta)\widehat{v}_p(\theta)e(k\theta)d\theta
  \end{equation}
  for all $k\in\Zz$.
  
  For any $k\in\Zz$, we have
  $$
  \sup_{p} \frac{1}{p}\Bigl|\sum_{0\leq n<p}t_p(n) \varpi(n+k)\Bigr| \leq
  \sup_{p} \frac{1}{p}\Bigl|\sum_{0\leq n<p}\varpi(n+k)\Bigr| +\Bigl(
  \sum_{p}\Bigl|\frac{1}{p}\sum_{0\leq n<p}\tau_p(n)
  \varpi(n+k)\Bigr|^2 \Bigr)^{1/2}
  $$
  (where $p$ always ranges over $\sfP$). Hence
  \begin{multline*}
    \|\widetilde{M}(\varpi)\|_2^2= \sum_{k\in \Zz} \sup_{p}
    \frac{1}{p}\Bigl|\sum_{0\leq n<p}t_p(n) \varpi(n+k)\Bigr|^2
    \\
    \leq 2\sum_{k\in \Zz} \sup_{p} \frac{1}{p}\Bigl|\sum_{0\leq
      n<p}\varpi(n+k)\Bigr|^2 +2\sum_{k\in \Zz} \sum_p
    \Bigl|\frac{1}{p}\sum_{0\leq n<p}\tau_p(n) \varpi(n+k)\Bigr|^2.
  \end{multline*}
  The first expression is $\leq C'_3\|\varpi\|_2^2$ by the classical
  maximal ergodic theorem in $L^2$ for functions on $\Zz$
  (see~\cite[\S\,2.6]{einsiedler-ward}). By~(\ref{eq-convol}) and the
  Plancherel formula, we estimate the second one as follows:
  \begin{align*}
    \sum_{k\in \Zz} \sum_p
    \Bigl|\frac{1}{p}\sum_{0\leq n<p}\tau_p(n) \varpi(n+k)\Bigr|^2
    &=\sum_{k\in\Zz}\sum_p
      \Bigl|\int_{\Rr/\Zz}\widehat{\varpi}(\theta)
      \widehat{v}_p(\theta)e(k\theta)d\theta\Bigr|^2\\
    &=\sum_p
      \int_{\Rr/\Zz}|\widehat{\varpi}(\theta)
      \widehat{v}_p(\theta)|^2d\theta\leq \Bigl(\sum_p
      \|\widehat{v}_p(\theta)\|_{\infty}^2\Bigr)\|\varpi\|_2^2.
  \end{align*}
  Applying Corollary~\ref{cor-ft}, we have
  $$
  \sum_p
  \|\widehat{v}_p(\theta)\|_{\infty}^2
  \ll \sum_{p\in\sfP}\frac{(\log p)^2}{p}=s(\sfP),
  $$
  where the implied constant depends only on the conductor of
  $(\sheaf{F}_p)$, and the result follows.
\end{proof}

\section{Pointwise ergodic theorem}
\label{sec-pointwise}

We give in this section a second proof of Theorem~\ref{th-main}, (2),
for almost Fourier families, arguing using a transfer principle as in
the previous section. This is obviously more complicated than our first
proof, but it is interesting that the sparseness condition which arises
is the same as before.

We consider a dynamical system $(X,\mu,f)$ and a family of sheaves
$(\sheaf{F}_p)$ with bounded conductor as in the previous section,
with trace functions $t_p$, defined for~$p$ in a sparse set~$\sfP$. We
assume that the family is almost Fourier (Definition~\ref{def-af}) of
mean $r\geq 0$.

\begin{proposition}\label{pr-pointwise}
  Let $\varphi\in L^2(X,\mu)$.  Then
  $$
  \frac{1}{p}\sum_{0\leq n<p}t_p(n)\ \varphi(f^n(x))
  $$
  converges for $\mu$-almost all $x\in X$. If $r=0$, or in other
  words, if all sheaves $\sheaf{F}_p$ are Fourier sheaves, or if
  $(X,\mu,f)$ is weakly mixing, then the limit is zero.
\end{proposition}

For the proof, we reduce to the shift by means of an intermediate
inequality. For a function $\varpi$ on $\Zz$, we write as before
$$
u_{p}(\varpi)(k)=\frac{1}{p}\sum_{0\leq n<p}t_p(n) \varpi(n+k).
$$

\begin{lemma}
Assume that for any infinite subset $\mathsf{Q}\subset \sfP$,
there exists a constant $C_4$ such that, for any function $\varpi$ on
$\Zz$ with bounded  support, we have
$$
\sum_{\ell\in\mathsf{Q}} \Bigl\|\,
\sup_{\ell<p<\ell^+}|u_{p}(\varpi)-u_{\ell^+}(\varpi)|\,
\Bigr\|_2^2\leq C_4\|\varpi\|^2,
$$
where $\ell^+$ is the element following $\ell$ in the subset
$\mathsf{Q}$, and $p$ ranges over elements in $\sfP$.  Then
Proposition~\ref{pr-pointwise} holds.
\end{lemma}

\begin{proof}
This has two steps. First, in the same manner that
Lemma~\ref{lm-transfer} is proved, the statement, if it holds, implies
the corresponding bound
$$
\sum_{\ell\in\mathsf{Q}} \Bigl\|\,
\sup_{\ell<p<\ell^+}|u_{p}(\varphi)-u_{\ell^+}(\varphi)|\,
\Bigr\|_2^2\leq C_4\|\varphi\|^2,
$$
for any $\varphi\in L^2(X,\mu)$, for any infinite subset $\mathsf{Q}$.

Next, one argues by contradiction that this last set of bounds, for a
given $\varphi$, implies that $u_p(\varphi)$ converges $\mu$-almost
everywhere.
\end{proof}

Finally, we prove the auxiliary bounds. 

\begin{proposition}
  Let $\mathsf{Q}\subset \sfP$ be an infinite subset.  There
  exists a constant $C_4$ such that, for any function $\varpi$ on
  $\Zz$ with bounded support, we have
  $$
  \sum_{\ell\in\mathsf{Q}} \Bigl\|\,
  \sup_{\ell<p<\ell^+}|u_{p}(\varpi)-u_{\ell^+}(\varpi)|\,
  \Bigr\|_2^2\leq C_4\|\varpi\|^2.
  $$
\end{proposition}

\begin{proof}
  Writing $t_p(n)=r+\tau_p(n)$, where $\tau_p(n)$ is the trace
  function of a Fourier sheaf of bounded conductor, and applying the
  known behavior from the standard pointwise ergodic theory to the
  first term, we are reduced to showing that
  $$
  \sum_{\ell\in\mathsf{Q}} \Bigl\|\,
  \sup_{\ell<p<\ell^+}|\nu_{p}(\varpi)-\nu_{\ell^+}(\varpi)|\,
  \Bigr\|_2^2\leq C_5\|\varpi\|^2.
$$
for some constant $C_5$, where $\nu_p$ is the averaging operator for
the trace function $\tau_p$. The left-hand side of the inequality is
equal to 
$$
  \sum_{\ell\in\mathsf{Q}} \sum_{k\in\Zz}
  \Bigl(\sup_{\ell<p<\ell^+}|\nu_{p}(\varpi)(k)-
  \nu_{\ell^+}(\varpi)(k)|\Bigr)^2
  \ll
  \sum_{\ell\in\mathsf{Q}} \sum_{k\in\Zz}
  \sum_{\ell<p<\ell^+}|\nu_{p}(\varpi)(k)|^2
  + \sum_{\ell\in\mathsf{Q}} \sum_{k\in\Zz}
  |\nu_{\ell^+}(\varpi)(k)|^2
$$
where the implied constant is absolute. The first sum here is larger
than the second, and it is at most
$$
\sum_{p\in\sfP} \sum_{k\in\Zz} |\nu_{p}(\varpi)(k)|^2 =
\sum_{p\in\sfP}\int_{\Rr/\Zz}
|\widehat{\varpi}(\theta)\widehat{\nu}_p(\theta)|^2d\theta
$$
by the Plancherel formula and~(\ref{eq-convol}). Using
Corollary~\ref{cor-ft}, we obtain the desired bound.
\end{proof}

\section{Is sparseness necessary?}

It is now natural to ask whether the restriction to sparse sets of
primes necessary in the maximal and pointwise ergodic theorems, or not.

The first remark is that, for a classical (even weighted) sequence of
ergodic averages
$$
u_N(x)=\frac{1}{N}\sum_{0\leq n<N} w(n)\varphi(f^n(x)),
$$
convergence along sparse sequences of $N$ implies convergence of the
whole sequence. For instance, assume that there is convergence to~$0$
for~$N$ growing at least like a geometric progression with ratio
$1+\delta>0$, and assume that~$w$ and $\varphi$ are bounded. For an
arbitrary~$N\geq 1$, pick~$M\geq 1$ such that $M\leq
N<(1+\delta)M$. We obtain an obvious upper bound
$$
|u_N|\leq |u_M|+\frac{C\delta M}{N}\leq |u_M|+\delta C
$$
for some constant $C\geq 0$, so that
$$
\limsup_{N\to+\infty} |u_N|\leq \delta C,
$$
and if this holds for any $\delta>0$, we obtain $u_N\to 0$.  Here the
key point is that the restriction of the weight $w(n)$ to a shorter
interval is the same as the weight used for the average over that
interval -- this property fails for ``triangular'' averages like those
appearing in our situation.

Here is an abstract example which could be a guide to an example where
almost everywhere convergence is \emph{not true} in our
setting.\footnote{ This is related to the well-known fact that
  convergence almost everywhere is not convergence with respect to any
  topology.}  Let $X$ be the product over primes $\ell$ of copies of
$\Rr/\Zz$, viewed as a compact topological group and as a probability
space with its Haar measure~$\mu$. For $\ell$ prime, fix an arbitrary
measurable subset~$A_{\ell}\subset\Rr/\Zz$ with
measure~$(\log\ell)^2/\ell$ (in~$\Rr/\Zz$).

Now, for $p$ prime, let $\varphi_{p}$ be the characteristic function
of the set $Y_{p}\subset X$ of all $(\theta_{\ell})\in X$ such that
the $p$-component $\theta_{p}$ belongs to~$A_p$.
Thus $\mu(Y_{p})=(\log p)^2/p$.

We claim that:
\begin{enumerate}
\item the sequence $(\varphi_{p})$ does \emph{not} converge almost
  everywhere;
\item but, for any \emph{sparse} set of primes $\sfP$, the subsequence
  $(\varphi_{p})_{p\in\sfP}$ converges almost everywhere to $0$.
\end{enumerate}

Indeed, the first assertion results from the independence of the
functions $\varphi_{p}$ (in probabilistic terms, they are
independent random variables on~$X$) and from the non-trivial
direction of the Borel-Cantelli lemma, since
$$
\sum_{p}\mu(Y_p)= \sum_{p}\frac{(\log p)^2}{p}=+\infty,\quad\quad
\sum_{p}\mu(X\setminus Y_p)=\sum_{p}\Bigl(1-\frac{(\log
  p)^2}{p}\Bigr)=+\infty,
$$
which shows that for almost all $\theta=(\theta_{\ell})$ in $X$, we
have $\theta\in Y_{p}$ (resp. $\theta\notin Y_{p}$) for infinitely
many $p$, so both~$\varphi_{p}(\theta)=0$ and~$\varphi_{p}(\theta)=1$
occur infinitely often.  

The second assertion results from the easy direction of the
Borel-Cantelli lemma, which implies that if $\sfP$ is a sparse set of
primes, then $\mu$-almost every element $\theta=(\theta_{\ell})\in X$
belongs only to finitely many $Y_{p}$ for $p\in\sfP$, so
that~$\varphi_p(\theta)=0$ for all~$p$ large enough in~$\sfP$.

The question is now whether such a model situation can arise in
ergodic averages with trace functions (of sheaves with bounded
conductor).  Roughly speaking, this would amount to having a dynamical
system $(X,\mu,f)$ and a function $\varphi$ on $X$ such that
$$
\frac{1}{p}\sum_{0\leq n<p}t_p(n)\, \varphi(f^n(x))\to
\begin{cases}
  1& \text{ with  probability $(\log p)^2/p$},\\
  0&\text{ with probability $1-(\log p)^2/p$},
\end{cases}
$$
\emph{and} the respective sets of $x$ where these limits hold should
be asymptotically independent enough to apply the Borel-Cantelli lemma
(exact independence is not necessary, e.g., a sufficient amount of
pairwise independence suffices, as in the Erd\H os-R\'enyi version of
the Borel-Cantelli theorem, see~\cite[\S\,1]{er}). (Moreover, the
limits could obviously be different, it is enough that the two
possibilities be separated enough that both occuring infinitely often
excludes convergence).

It does not seem impossible to have such a configuration, especially
since the trace function is \emph{a priori} ours to select, with the
condition that the conductors remain bounded, which might make it
possible to exploit the frequent rough independence of primes.

\begin{remark}\label{rm-sparse}
  (1) We would also show that convergence does not hold almost surely
  if the ergodic average converges to $1$ with probability $1/p$
  (instead of $(\log p)^2/p$), which would allow for convergence over
  all sets of primes with
  $$
  \sum_{p\in\sfP}\frac{1}{p}<+\infty.
  $$
  \par
  This configuration is maybe more likely to be possible.
  \par
  (2) If we have a system where the ergodic averages converge
  \emph{everywhere} for all sparse subsets of the primes, then they
  converge everywhere. (Indeed, the limit $\psi$ would have to be
  independent of the sparse subset, since the union of two sparse sets
  is sparse, and then by contraposition, if the sequence was not
  convergent to $\psi$, some subsequence would avoid a fixed
  neighborhood of $\psi$, and some further subsequence would be
  sparse.)
\end{remark}

The following is currently the closest example that we know. It doesn't
quite address the main question, since it involves non-Fourier sheaves
and systems with non-trivial Kronecker factors.

Let $X=(\Rr/\Zz)^2$ (viewed as column vectors) with the Haar measure
$\mu$. Let $f(x,y)=(x+y,y)$, so that $f$ is the action of an
$\SL_2(\Zz)$-matrix, and therefore preserves $\mu$.  For $(x,y)\in X$,
we have
$$
f^n(x,y)=(x+ny,y).
$$

Define $\varphi\colon X\to \Cc$ by $\varphi(x,y)=e(x)$. The ergodic
averages are therefore
$$
\frac{1}{p}\sum_{0\leq n<p}t_p(n)\varphi(f^n(x,y))=
\frac{e(x)}{p}\frac{\sin(\pi p(y-a_p/p))}{\sin(\pi (y-a_p/p))}
e\Bigl(\frac{(p-1)}{2}(y-a_p/p)\Bigr).
$$

\begin{lemma}\label{lm-single}
  There exists a sequence $(a_p)_p$ of integers such that $0\leq a_p<p$
  for all primes~$p$, with the following property: for almost all
  $\theta\in\Rr/\Zz$, there exist infinitely many $p$ such that
  $|\theta-a_p/p|\leq 1/(100p)$.
\end{lemma}

\begin{proof}
  Here is one quick proof using fairly standard (but non-trivial) facts
  about the distribution of primes.  Another more elementary argument is
  explained in the note~\cite{single} for the simple proof, which also
  has some more discussion of this somewhat unusual diophantine
  approximation statement.

  Let $\mathcal{A}$ be the product over primes of the sets $\{0,\ldots,
  p-1\}$; it is a probability space with the product of the uniform
  probability measures.

  Let $c=1/100$ (any other positive constant would work).  For any prime
  $p$ and $\uple{a}\in\mathcal{A}$, we write
  $I_p(\uple{a})=[a_p/p-c/p,a_p/p+c/p]$, viewed as random intervals
  on~$\mathcal{A}$. Let $x\in [0,1]$. We then have
  $$
  \proba(x\in I_p)=\frac{1}{p}\sum_{\substack{0\leq a<p\\|x-a/p|<c/p}}1
  $$
  and hence $\proba(x\in I_p)$ is either~$0$ or~$1/p$, depending on
  whether there exists an integer~$a$ such that the fractional part
  of~$xp$ is~$<c$, or not.
  
  It is known that if~$x$ is irrational, then we have
  \begin{equation}\label{eq-vino}
    \sum_{\{xp\}<c}\frac{1}{p}=+\infty
  \end{equation}
  (precisely, this follows by summation by parts from the more precise
  results, first proved by Vinogradov, which give an asymptotic formula
  with main term $c\pi(X)$ for the number of primes $p\leq X$
  satisfying~$\{xp\}<c$, as $X\to +\infty$;
  see~\cite[Ch.\,XI]{vinogradov}, and note that this result has been
  improved and simplified since then). Thus, since the events
  $\{x\in I_p\}$ are independent by construction, the Borel--Cantelli
  Lemma implies
  $$
  \proba(x\in I_p\text{ for infinitely many } p)=1
  $$
  for any irrational~$x$.
  
  Now by Fubini's Theorem, we obtain
  \begin{align*}
    \expect(\lambda(A_{\uple{a}}))&=
    \expect\Bigl(
    \int_{0}^1 1_{\{x\in I_p\text{ for infinitely many } p\}}\ dx
    \Bigr)\\
    &=\int_0^1 \proba(x\in I_p\text{ for infinitely many } p)dx
    =1,
  \end{align*}
  and since $\lambda(A_{\uple{a}})\leq 1$, this means that $A_{\uple{a}}$
  has measure~$1$ for almost all sequences~$(a_p)$. 
\end{proof}

Now fix a sequence $(a_p)$ as given by that lemma and define
$t_p(n)=e(-a_pn/p)$. These are trace functions of Artin-Schreier sheaves
with bounded conductor. Let $\mathsf{P}$ be any set of primes with
$$
\sum_{p\in\mathsf{P}}\frac{\log p}{p}<+\infty.
$$

Then, for almost all~$(x,y)$, we have
$$
\Bigl|y-\frac{a_p}{p}\Bigr|\geq \frac{\log p}{p}
$$
for all but finitely many $p\in\mathsf{P}$, by the easy Borel--Cantelli
lemma, hence
$$
\frac{1}{p}\sum_{0\leq n<p} t_p(n)\varphi(f^n(x,y))\to 0
$$
almost surely along~$\mathsf{P}$. (And note that sparseness could be
measured with $\log p$ replaced by any function tending to infinity
with~$p$.)

On the other hand, for almost all $(x,y)\in X$, the properties of the
sequence $(a_p)$ prove that there exists a subsequence of primes for which
$$
\frac{1}{p}\sum_{0\leq n<p} t_p(n)e(n\theta)\gg 1,
$$
hence for which
$$
\frac{1}{p}\sum_{0\leq n<p} t_p(n)\varphi(f^n(x,y))
$$
does \emph{not} converge to~$0$ along the primes. Since the result for
sparse sequences mean that this sequence could only converge to~$0$
almost surely, we conclude that the ergodic averages
$$
\frac{1}{p}\sum_{0\leq n<p} t_p(n)\varphi\circ f^n
$$
do not converge almost surely.

\section{Questions}

The following further natural questions arise from this note:

\begin{enumerate}
\item Are there maximal and pointwise ergodic theorems with trace
  functions for $\varphi \in L^p$ where $p\not=2$, especially for $p=1$?
  For $p>1$, one can certainly expect to be able to prove theorems in
  $L^p$ by adapting the ideas of Bourgain~\cite{bourgain2}.
  The case $p=1$ might well be the most interesting; we recall here that
  Buczolich and Mauldin~\cite{b-m} have proved that there is no maximal
  or pointwise ergodic theorem in $L^1$ for averages along the squares
  (see also LaVictoire's generalization of this fact~\cite{victoire},
  which relies on non-trivial arithmetic information).
\item Are there similar results for ``classical non-conventional
  averages'' with trace functions, such as
  $$
  \frac{1}{p}\sum_{0\leq n<p} t_p(n)\ (\varphi\circ f^n)\ (\varphi\circ
  f^{2n}) \cdots (\varphi\circ f^{kn})
  $$
  (where $k$ is fixed; these occur without weights in Furstenberg's
  approach to Szemerédi's Theorem, see~\cite[Ch. 7]{einsiedler-ward})
  or
  $$
  \frac{1}{p}\sum_{0\leq n<p} t_p(n)\ \varphi\circ f^{n^2},
  $$
  and other polynomials in place of $n^2$? The versions without weights
  are parts of Bourgain's celebrated work~\cite{bourgain1,
    bourgain2,bourgain3}.

  The first type of averages is intriguing, if only because trace
  functions are known to satisfy a very strong from of the inverse
  theorem for Gowers norms (by work of Fouvry, Kowalski and
  Michel~\cite{fkm-gowers}).
\item Maybe most important: are there interesting applications of such
  ergodic averages?
\end{enumerate}

\end{document}